\documentclass[a4paper,11pt]{article}
\usepackage[utf8]{inputenc}

\usepackage{amsmath,amssymb,amsthm,mathtools,authblk,dsfont}

\usepackage{xcolor}
\usepackage{hyperref}
\hypersetup{colorlinks=true, breaklinks=true, urlcolor= blue, linkcolor= blue, citecolor=orange,linktocpage}

\usepackage[capitalize,noabbrev]{cleveref}
\crefname{equation}{equation}{equations}
\Crefname{equation}{Equation}{Equations}

\theoremstyle{plain}
\newtheorem{thm}{Theorem}
\newtheorem{prp}[thm]{Proposition}
\newtheorem{cor}[thm]{Corollary}
\newtheorem{lem}[thm]{Lemma}
\numberwithin{thm}{section}

\theoremstyle{example}
\newtheorem*{expl}{Example}

\theoremstyle{definition}
\newtheorem*{nota}{Notation}
\newtheorem{rem}[thm]{Remark}

\newcommand*{\proofstepname}{Proof step}
\newenvironment{proofstep}[1][\proofstepname]{\begin{proof}[#1]}{\end{proof}}

\usepackage{geometry}
\title{A high-genus asymptotic expansion of \\ Weil--Petersson volume polynomials}

\author[$\dagger$]{Nalini Anantharaman}
\author[$\dagger$]{Laura Monk}

\affil[$\dagger$]{{\small Universit\'e de Strasbourg, CNRS, IRMA UMR 7501, F-67000 Strasbourg, France}}

\date{\today}

\newcommand{\x}{\mathbf{x}}
\newcommand{\R}{\mathbb{R}}
\newcommand{\N}{\mathbb{N}}

\renewcommand{\d}{\, \mathrm{d}}
\renewcommand{\O}[2][]{\mathcal{O}_{#1} \left( #2 \right)}
\newcommand{\1}[1]{\mathds{1}_{#1}}

\let\div\relax
\newcommand{\div}[1]{\paren*{\frac{#1}{2}}}
\newcommand{\cd}[1]{\mathrm{c} \paren{#1}}
\newcommand{\sd}[1]{\mathrm{sc} \paren{#1}}

\DeclarePairedDelimiter{\abso}{|}{|}
\DeclarePairedDelimiter{\paren}{(}{)}
\DeclarePairedDelimiter{\brac}{[}{]}
\DeclarePairedDelimiter{\floor}{\lfloor}{\rfloor}

\DeclarePairedDelimiter{\jap}{\langle}{\rangle}

\newcommand{\volwp}{\mathrm{Vol}^{\mathrm{\scriptsize{WP}}}}

\DeclareMathOperator{\sinhc}{sinhc}

\newcommand{\z}[1]{\mathbf{0}^{#1}}

\begin{document}

\maketitle

\begin{abstract}
  The object under consideration in this article is the total volume $V_{g,n}(x_1, \ldots, x_n)$ of the moduli space of
  hyperbolic surfaces of genus $g$ with $n$ boundary components of lengths $x_1, \ldots, x_n$, for the Weil--Petersson
  volume form. We prove the existence of an asymptotic expansion of the quantity $V_{g,n}(x_1, \ldots, x_n)$ in terms of
  negative powers of the genus $g$, true for fixed $n$ and any $x_1, \ldots, x_n \geq 0$. The first term of this
  expansion appears in work of Mirzakhani and Petri (2019), and we compute the second term explicitly. The main tool used in
  the proof is Mirzakhani's topological recursion formula, for which we provide a comprehensive introduction.
\end{abstract}

\section{Introduction and statement of the results}

\subsection{First definitions and notations}

Let $g$, $n$ be two integers such that $2g-2+n>0$. For any given length
vector $\x = (x_1, \ldots, x_n) \in \R_{\geq 0}^n$, we define the moduli
space $\mathcal{M}_{g,n}(\x)$ to be the set of isometry classes of
surfaces $X$ satisfying the following:
\begin{itemize}
\item $X$ is a connected oriented hyperbolic surface of genus $g$ with $n$ labelled
  boundary components $b_1, \ldots, b_n$,
\item for all $i \in \{1, \ldots, n\}$, the boundary components
  $b_i$ is a closed geodesic of length $x_i$ if $x_i > 0$, or a
  cusp if $x_i=0$.
\end{itemize}
This space is an orbifold of dimension $6g-6+2n$. It is
equipped with a natural symplectic form, the Weil--Petersson form
$\omega^{\mathrm{WP}}_{g,n,\x}$ \cite{weil1958,goldman1984}, which
canonically induces a volume form
\begin{equation*}
  \volwp_{g,n,\x} :=
  \frac{1}{(3g-3+n)!} \underbrace{\omega^{\mathrm{WP}}_{g,n,\x} \wedge \ldots \wedge
  \omega^{\mathrm{WP}}_{g,n,\x}}_{3g-3+n \text{ times}}.
\end{equation*}
The object under consideration in this article is the total volume of the
moduli space, 
\begin{equation*}
  V_{g,n}(\x) : =\volwp_{g,n,\x}(\mathcal{M}_{g,n}(\x)) < + \infty.
\end{equation*}
By work of Mirzakhani, \cite{mirzakhani2007}, the volume $V_{g,n}(\x)$ is a symmetric polynomial function in
$x_1^2, \ldots, x_n^2$ of degree $3g-3+n$, and can therefore be written as
\begin{equation*}
  V_{g,n}(\x)
  = \sum_{\substack{\alpha_1 + \ldots + \alpha_n \\\leq 3g-3+n}} c_{g,n}(\alpha)
  \prod_{j=1}^n \frac{x_j^{2\alpha_j}}{2^{2\alpha_j}(2\alpha_j+1)!}
\end{equation*}
for a family of coefficients $(c_{g,n}(\alpha))_\alpha$\footnote{For the
  sake of readability, our notation differs from the usual notation
  $[\tau_{\alpha_1} \ldots \tau_{\alpha_n}]_{g,n}$ from intersection theory
  (see \cite{mirzakhani2007a}).}.

\subsection{Asymptotic expansions of Weil--Petersson volumes}

We shall provide an asymptotic expansion of the quantity $V_{g,n}(\x)$,
true for a fixed $n \geq 1$, any length vector $\x \in \R_{\geq 0}^n$, and as
the genus $g$ approaches infinity. The motivations for this question and
this particular choice of setting are presented in
\cref{sec:motiv-study-rand}.

\paragraph{Notations}

Let $n \geq 1$ be an integer. We will use the $\ell^1$ and $\ell^\infty$ norms on~$\R^n$, denoted as $|\cdot|$ and
$|\cdot|_\infty$ respectively.  For any real number $x$, we define $\jap{x} := \sqrt{1+x^2}$. We extend this definition
to~$\mathbf{x} \in \R^n$ by setting $\jap{\mathbf{x}} := \jap{\abso{\mathbf{x}}}$.

We let $\N_0 := \{0, 1, 2, \ldots \}$ denote the set of non-negative integers. We write
$\z{n} := (0, \ldots, 0) \in \N_0^n$.  For any $1 \leq i \leq n$, $\delta_i$ denotes the discrete derivation w.r.t. the
$i$-th coordinate, acting on functions $v : \N_0^n \rightarrow \R$ by
\begin{equation*}
  \delta_i v(\alpha) := v(\alpha)
  - v(\alpha_1, \ldots, \alpha_{i-1}, \alpha_i+1, \alpha_{i+1}, \ldots, \alpha_n).
\end{equation*}
We will use the usual conventions for multi-indices $\alpha \in \N_0^n$, and notably:
\begin{align*}
  & (\alpha_1, \ldots, \hat{\alpha}_j, \ldots, \alpha_n)
  := (\alpha_1, \ldots, \alpha_{j-1}, \alpha_{j+1}, \ldots, \alpha_n) \in \N_0^{n-1} \\
  & \partial^{\alpha} := \partial_1^{\alpha_1} \ldots \partial_n^{\alpha_n} \\
  & \alpha_I := (\alpha_{i_1}, \ldots, \alpha_{i_r})
             \quad \text{for } I = \{i_1 < i_2 < \ldots i_r\} \subset \{1, \ldots, n\}.
\end{align*}

We write $A = \O{B}$ if there exists a universal constant $C>0$ such that, for any choice of parameters,
$\abso{A} \leq C B$. If the constant depends on a parameter~$p$, we rather write $A = \O[p]{B}$.

The function $\sinhc: \R \rightarrow \R$ is defined by
\begin{equation*}
  \sinhc (x) :=
  \begin{cases}
    \frac{\sinh x}{x} & \text{if } x \neq 0 \\
    1 & \text{otherwise.}
  \end{cases}
\end{equation*}

\paragraph{State of the art}

The value at zero of the Weil--Petersson volume, denoted as $V_{g,n} := V_{g,n}(\z{n}) = c_{g,n}(\z{n})$ and which
corresponds to the case where all boundary components are cusps, has been thoroughly studied. In \cite{mirzakhani2015},
Mirzakhani and Zograf have proved it admits a full asymptotic expansion of the form
\begin{equation*}
  V_{g,n} = C_V \, \frac{(2g-3+n)!(4 \pi^2)^{2g-3+n}}{\sqrt{g}}
  \left(1 + \frac{e_n^{(1)}}{g} + \ldots + \frac{e_n^{(N)}}{g^N}
    + O_n \left( \frac{1}{g^{N+1}} \right) \right),
\end{equation*}
where $C_V>0$ is a universal constant.  Asymptotic expansions of other quantities, such as the coefficient
$c_{g,n}(\alpha)$ for a fixed $\alpha$, are also provided.

Unfortunately, for general length vectors $\x \in \R_{\geq 0}^n$, the best approximation of $V_{g,n}(\x)$ in literature
so far is the first-order approximation proved by Mirzakhani and Petri \cite[Proposition
3.1]{mirzakhani2019}\footnote{Actually, the factor $|\x|$ in the remainder is missing in \cite{mirzakhani2019}. This minor error
  has no implication for the purposes of Mirzakhni and Petri's article, or the further applications 
  \cite{wu2021,lipnowski2021}, but would have contradicted our second-order expression (\cref{thm:order_two}).}, which states that for any $\x$,
\begin{equation}
  \label{eq:first_order}
  \frac{V_{g,n}(\x)}{V_{g,n}}
    =  \prod_{j=1}^n \sinhc \div{x_j}
  +  \O[n]{\frac{\abso{\x}}{\jap{g}} \exp\div{x_1 + \ldots + x_n}}.
\end{equation}
This estimate plays a key role in \cite{wu2021,lipnowski2021}, as we will see in \cref{sec:motiv-study-rand}. The aim of
this article is to prove a similar result, with an error term decaying like $1/\jap{g}^{N+1}$ for arbitrarily large $N$
rather than $N=0$.

\paragraph{Statement of the main result}

The main result proved in this article is the following.

\begin{thm}
  \label{theo:volume_asympt_exp}
  For any integers $g \geq 0$, $n \geq 1$ such that $2g-2+n>0$, there exists a family of $n$-variable even polynomial
  functions $(P_{g,n}^{(N,I_\pm)})_{N, I_\pm}$, for $N \geq 0$ and $I_+ \sqcup I_- \subseteq \{1, \ldots, n\}$, such
  that for any integer $N \geq 0$ and any length vector $\x \in \R_{\geq 0}^n$,
  \begin{equation}
    \label{e:thm_main}
    \frac{V_{g,n}(\x)}{V_{g,n}}
    = F_{g,n}^{(N)}(\x) 
    + \O[N,n]{\frac{\jap{\x}^{3N+1}}{\jap{g}^{N+1}} \exp \div{x_1+\ldots
        +x_n}}
  \end{equation}
  where
  \begin{equation}
    \label{eq:1}
    F_{g,n}^{(N)}(\x)
    := \sum_{I_+ \sqcup I_- \subseteq \{1, \ldots,n \}}
    P_{g,n}^{(N,I_\pm)}(\x) \prod_{i \in I_+} \cosh
    \div{x_i}
    \prod_{i \in I_-} \sinhc \div{x_i}.
  \end{equation}
  Furthermore, there exists constants $D_{n,N}, A_N \geq 0$ such that the polynomial function $P_{g,n}^{(N,I_\pm)}$ can
  be expressed as a polynomial of degree $\leq D_{n,N}$, and its coefficients can be written as linear combinations
  (independent of $g$) of the $c_{g,n}(\alpha)/V_{g,n}$ for multi-indices $\alpha$ such that
  $\abso{\alpha}_\infty \leq A_N$.
\end{thm}

\begin{rem}
  \label{rem:deg}
  More precisely, our proof shows that degree of $P_{g,n}^{(N,I_\pm)}$ seen as a polynomial function of the variables
  $(x_i)_{i \in I_- \cup I_+}$ is $\leq 2N$, while as a polynomial function of $x_i$ for a $i \notin I_+ \cup I_-$ it is
  strictly smaller than the constant~$a_{N+1}$ from \cref{thm:derivative_coeff}, of size discussed below. In particular,
  one can take $D_{n,N}$ to be $2N+n(a_{N+1}-1)$. The value of $A_N$ provided by the proof is $2N+a_{N+1}$.
  
  It should be noted that, for any integer $i$, the dependency of the function $F_{g,n}^{(N)}(\x)$ with respect to a
  large $x_i$ will be dominated by the terms $I_+ \sqcup I_- \subseteq \{1, \ldots, n\}$ of the sum \eqref{eq:1} for which
  $i \in I_+ \cup I_-$, because they behave exponentially rather than polynomially. As a consequence, the fact that our
  degree bound is weaker for indices $i \notin I_+ \cup I_-$ has little to no consequences on the behaviour of
  $F_{g,n}^{(N)}(\x)$ for large values of~$\x$.
\end{rem}

\paragraph{Coefficient estimate and sketch of the proof}

The key technical step to prove \cref{theo:volume_asympt_exp} is an estimate for the discrete derivatives of
$\alpha \mapsto c_{g,n}(\alpha)$.

\begin{thm}
  \label{thm:derivative_coeff}
  For any order $N \geq 0$, there exists a constant $a_N \geq 0$ satisfying the following. For any integers $g \geq 0$,
  $n \geq 1$ such that $2g-2+n>0$, and any multi-indices~$\mathbf{m}, \alpha \in \N_0^n$ such that
  $\abso{\mathbf{m}} \in \{2N-1,2N\}$ and $\alpha_i \geq a_N$ for every index $i$ such that $m_i>0$, 
  \begin{equation}
    \delta^{\mathbf{m}}c_{g,n}(\alpha)
    = \O[n,N]{\jap{\alpha}^{N} \frac{V_{g,n}}{\jap{g}^N}}.
  \end{equation}
\end{thm}

By a discrete Taylor-expansion result (\cref{lemm:taylor_shift}), \cref{thm:derivative_coeff} implies that the
coefficients $c_{g,n}(\alpha)$ can be well-approximated by functions which are almost polynomial in $\alpha$, and
\cref{theo:volume_asympt_exp} then follows.

Interestingly, the fact that $c_{g,n}(\alpha)$ can be approximated by functions which are almost polynomial in $\alpha$
had already been observed by Mirzakhani and Zograf in \cite[Lemma 4.8]{mirzakhani2015}. However, since the dependency on
$\alpha$ of the coefficients $c_{g,n}(\alpha)$ is not the main objective in \cite{mirzakhani2015}, the proof of this
statement is only sketched, and presented as a technical lemma. To the contrary, thanks to our new idea of estimating
the discrete derivatives of $\alpha \mapsto c_{g,n}(\alpha)$, our proof is fairly elementary. It only relies on one
application of Mirzakhani's topological recursion formula \cite{mirzakhani2007} and a few classic volume estimates from
\cite{mirzakhani2013}, all of which are carefully presented in \cref{sec:prereq}.

The parameter $a_N$ present in \cref{thm:derivative_coeff} encapsulates the fact that the volume coefficients
$c_{g,n}(\alpha)$ take exceptional values for small multi-indices $\alpha$. This phenomenon is already mentionned in
\cite[Remark 4.3]{mirzakhani2015}, where it is referred to as a `boundary effect'. It is not an artefact of the proof,
and can be observed in both Mirzakhani and Zograf's remark and our explicit formula for the second-order term, \cref{thm:order_two}.

The constant $a_N$ provided by our proof grows like $2^N$. This value is not optimal, and its exponential behaviour
comes as a drawback of our new induction argument. In \cite[Lemma 4.8]{mirzakhani2015}, a much smaller value $a_N = 2N$
is obtained, but we have unfortunately not been able to achieve a linear value using our method.

\paragraph{Expansion in negative powers of $g$}

Using the expansion of $c_{g,n}(\alpha)/V_{g,n}$ in negative powers of
$g$ for a fixed multi-index $\alpha$ proved by Mirzakhani and Zograf
\cite[Theorem 4.1]{mirzakhani2015}, we can straightforwardly deduce from
\cref{theo:volume_asympt_exp} the following expansion, which is now
uniquely defined.

\begin{cor}
  \label{cor:exp_pow_g}
  Let $n \geq 1$ be an integer.  There exists a unique family
  $(f_n^{(k)})_{k \geq 0}$ of functions such that
  for any integer $N \geq 0$, any genus $g \geq 1$ and any length
  vector $\x \in \R_{\geq 0}^n$,
  \begin{equation}
    \label{e:fn_approx}
    \frac{V_{g,n}(\x)}{V_{g,n}}
    = \sum_{k=0}^N \frac{f_{n}^{(k)}(\x)}{g^k}
    + \O[N,n]{\frac{\jap{\x}^{3N+1}}{g^{N+1}} \exp \div{x_1 + \ldots + x_n}}.
  \end{equation}
  Furthermore, for any $k \geq 0$, the function $f_n^{(k)}$ can be expressed as
  \begin{equation}
    \label{eq:fn_sinh}
    f_n^{(k)}(\x)
    = \sum_{I_+ \sqcup I_- \subseteq \{1, \ldots, n\}}
    Q_n^{(k,I_\pm)}(\x) \prod_{i \in I_+} \cosh \div{x_i}
    \prod_{i \in I_-} \sinhc \div{x_i},
  \end{equation}
  where $Q_n^{(k,I_\pm)}$ are uniquely defined even $n$-variable polynomial functions. 
\end{cor}

The symmetry of $V_{g,n}(\x)$ implies that, for all $k$, $f_n^{(k)}$ is
symmetric, which in turn provides some relations between the
$Q_n^{(k,I_\pm)}$ for $I_+ \sqcup I_- \subseteq \{1, \ldots, n\}$.

\paragraph{Explicit expression for the first orders}

By \cite[Proposition 3.1]{mirzakhani2019}, the value of the first
approximation $f_n^{(0)}$ is
\begin{equation*}
  f_n^{(0)}(\x) = \prod_{j=1}^n \sinhc \div{x_j}.
\end{equation*}
We provide an explicit expression for the second-order approximation.
In order to simplify the notations, we introduce the functions
$\mathrm{c}$, $\mathrm{s}$ defined by
\begin{equation*}
  \forall x, \quad \cd{x} := \cosh \div{x} \quad \text{and} \quad
  \sd{x} := \sinhc \div{x}.
\end{equation*}
Then, the second-order expansion can be written as follows.

\begin{thm}
  \label{thm:order_two}
  For any $n \geq 1$ and $\x \in \R_{\geq 0}^n$,
  \begin{multline*}
    f_{n}^{(1)}(\x)
    = \frac{1}{\pi^2} \sum_{i=1}^n
    \brac*{\cd{x_i} +1
      - \paren*{\frac{x_i^2}{16} + 2} \, \sd{x_i}}
    \prod_{k \neq i} \sd{x_k}  \\
    - \frac{1}{2\pi^2} \sum_{1 \leq i < j \leq n}
    \brac*{\cd{x_i} \, \cd{x_j} +1
      - 2 \, \sd{x_i} \, \sd{x_j}} 
    \prod_{k \notin \{i, j\}} \sd{x_k}.
  \end{multline*}
\end{thm}
Another formulation of this statement, using the notations of
\cref{theo:volume_asympt_exp}, can be found as \cref{prop:vgn_order_two}.

\begin{expl}
  For $n=1$, we obtain
  \begin{equation}
    \pi^2 f_1^{(1)}(x)
    = \cosh \div{x} + 1
      - \paren*{\frac{x}{8} + \frac{4}{x}} \sinh  \div{x}.
  \end{equation}
  For $n=2$, in the special case where $x_1=x_2$ (which often appears
  when using Mirzakhani's integration formula, see \cref{e:expl_mirz} for instance),
  \begin{equation}
    \pi^2 f_2^{(1)}(x,x)
    = \frac{2}{x} \sinh(x) 
    - \frac{12}{x^2} \sinh^2 \div{x}
    - \cosh^2 \div{x} + \frac{4}{x} \sinh \div{x}.
  \end{equation}
\end{expl}

\subsection{Motivation to the study of random compact hyperbolic
  surfaces}
\label{sec:motiv-study-rand}

The choice of the regime $g \gg 1$ while $n \geq 1$ is fixed is motivated by its great importance in the study of random
\emph{compact} hyperbolic surfaces of \emph{large genus}.

This topic has gained increasing popularity in recent years -- see
\cite{guth2011,mirzakhani2013,mirzakhani2019,monk2021b,nie2020,wu2021,lipnowski2021} for instance.  In these articles, the
surfaces are sampled using the Weil--Petersson probability measure $\mathbb{P}_g^{\mathrm{WP}}$, obtained by
renormalising the Weil--Petersson volume form on the moduli space~$\mathcal{M}_g$ of closed hyperbolic surfaces of genus
$g$. In particular, $n = 0$, which could appear to be contradictory since we assume in this article that~$n \geq 1$.

Actually, Weil--Petersson volumes $V_{g,n}(\x)$ for $n \geq 1$ and $\x \neq \z{n}$ appear systematically when using
Mirzakhani's integration formula \cite{mirzakhani2007}, the main tool available to compute expectations and
probabilities in the Weil--Petersson setting.  This is the reason why it is absolutely essential to understand such
volumes in order to study compact hyperbolic surfaces. For instance, 
\begin{equation}
  \label{e:expl_mirz}
  \mathbb{E}^{\mathrm{WP}}_g
  \brac*{\# \left\{
      \parbox{4.2cm}{$\gamma$ primitive simple closed \\
        geodesic, non-separating,  \\
        such that $a \leq \ell(\gamma) \leq b$}
    \right\}}
  = \int_a^b \frac{V_{g-1,2}(x, x)}{2 \, V_{g,0}} \, x \d x.
\end{equation}
In \cite{mirzakhani2019}, it is in order to estimate such quantities and prove the convergence of the number of
primitive closed geodesics of length $a \leq \ell \leq b$ to a Poisson law of parameter
$\lambda_{a,b} = \int_a^b \frac{2}{x} \sinh^2 \div{x} \d x$ as $g \rightarrow + \infty$, that Mirzakhani and Petri
compute the first-order approximation of $V_{g,n}(\x)$.

This first-order estimate has since then been used by Wu--Xue \cite{wu2021} and Lipnowski--Wright \cite{lipnowski2021}
in two independent proofs of the fact that the first non-zero eigenvalue $\lambda_1$ of the Laplace--Beltrami operator satisfies
\begin{equation}
  \label{e:3_16}
  \forall \epsilon > 0, \quad
  \lim_{g \rightarrow + \infty} \mathbb{P}_g^{\mathrm{WP}}
  \brac*{\lambda_1 \geq\frac{3}{16} -\epsilon}  = 1.
\end{equation}
Proving that \eqref{e:3_16} still holds if we replace the number $\frac{3}{16}$ by $\frac{1}{4}$, which would then be
optimal by \cite{cheng1975}, is a very active topic. This was achieved very recently for random covers of non-compact
surfaces by Hide and Magee \cite{hide2021}, but is still an open problem in the Weil--Petersson setting and for random
covers of compact surfaces.

As explained in \cite[Section 6.1.2]{monk2021a}, replacing $\frac{3}{16}$ by the `natural' next step, $\frac{2}{9}$,
requires amongst other things a second-order expansion such as \cref{thm:order_two}. Ultimately, we believe that
obtaining the optimal value~$\frac 14$ will require estimates with errors of size $1/g^N$ for arbitrarily large $N$, and
this is the core motivation behind this article.

\subsection{Organisation of the paper}

This article is organised as follows.
\begin{itemize}
\item In \cref{sec:prereq}, we review the different classic tools that are required to study the Weil--Petersson volume
  $V_{g,n}(\x)$. Notably, we provide a comprehensive introduction to the topological recursion formula satisfied by these
  functions proved by \cite{mirzakhani2007}, as well as a throughout proof of the first-order expansion from
  \cite{mirzakhani2019}.
\item In \cref{s:second_order}, we compute our new second-order expansion, \cref{thm:order_two}. This allows us to
  introduce a few notations and ideas that are useful to the proof of the higher-order expansion.
\item We then prove the estimate on the discrete derivatives $\delta^{\mathbf{m}}c_{g,n}(\alpha)$ of the volume
  coefficients, \cref{thm:derivative_coeff}, in \cref{sec:proof:coeff}. The proof proceeds by induction on the absolute
  value of the Euler characteristic $\abso{\chi} = 2g-2+n$, and the use of Mirzakhani's topological recursion formula.
\item Finally, we prove a shifted discrete Taylor expansion in \cref{sec:funct-ultim-polyn}. It allows us to
  approximate the coefficients $c_{g,n}(\alpha)$ by functions almost polynomial in $\alpha$, and hence conclude to
  \cref{theo:volume_asympt_exp} and \cref{cor:exp_pow_g}.
\end{itemize}

\paragraph{Acknowledgements}

The authors would like to thank Michael Lipnowski, St\'ephane Nonnenmacher, Bram Petri and Alex Wright for helpful
discussions. We furthermore thank the referee for their valuable comments on a previous version of this paper, which
motivated us to significantly improve the presentation of this new version.

\section{Preliminaries: Weil--Petersson volumes and Mirzakhani's topological recursion}
\label{sec:prereq}

In this section, we shall present some of the tools that are essential to the study of the total volume $V_{g,n}(\x)$ of
the moduli space of bordered hyperbolic surfaces $\mathcal{M}_{g,n}(\x)$. Notably, we will explain in detail
Mirzakhani's topological recursion formula proved in \cite{mirzakhani2007}, which allows to compute the volumes~$V_{g,n}(\x)$ recursively.

\subsection{Polynomial expression}

By \cite[Theorem 6.1]{mirzakhani2007}, the function $\x \mapsto V_{g,n}(\x)$ is a polynomial function that can be
written as
\begin{equation}
  \label{theo:mirz_coeff}
  V_{g,n}(\x) = \sum_{\abso{\alpha} \leq 3g-3+n} c_{g,n}(\alpha) \prod_{j=1}^n
  \frac{x_j^{2\alpha_j}}{2^{2\alpha_j}(2\alpha_j+1)!}\cdot
\end{equation}  
The polynomial $V_{g,n}(\x)$ is symmetric in the variables $x_1, \ldots, x_n$, and the coefficients $c_{g,n}(\alpha)$
are therefore invariant by permutation of the multi-index~$\alpha$. For convenience, we extend the definition of $c_{g,n}(\alpha)$
to any multi-index $\alpha \in \mathbb{Z}^n$, by setting it to be equal to zero unless it already defined by
\eqref{theo:mirz_coeff}.

The expression of the Weil--Petersson volume polynomial is known for surfaces of Euler characteristic $\chi = -1$,
i.e. for the pair of pants (of signature $(0,3)$) and the once-holed torus (of signature $(1,1)$). Indeed, there is only
one hyperbolic pair of pants with three boundary geodesics of prescribed lengths \cite[Theorem 3.1.7]{buser1992}, and
therefore $\mathcal{M}_{0,3}(\x)$ is reduced to an element and $V_{0,3}(\x)$ is the constant polynomial equal to
$1$. N\"{a}\"{a}t\"{a}nen and Nakanishi proved in \cite{naatanen1998} that for all $x \geq 0$,
\begin{equation*}
  V_{1,1}(x)  = \dfrac{\pi^2}{6} + \dfrac{x^2}{24}\cdot
\end{equation*}

The choice of the normalisation by $2^{2 \alpha_j}(2 \alpha_j+1)!$ in \cref{theo:mirz_coeff} is partly motivated by the
fact that it allows to interpret the coefficients $c_{g,n}(\alpha)$ as intersection numbers -- see \cite{mirzakhani2007a}. It furthermore simplifies the topological recursion formula that the
coefficients satisfy, which we shall now present.

\subsection{Mirzakhani's topological recursion formula}
\label{s:recu_form}

Whenever the number of boundary components $n$ is different from $0$, the volume polynomial $V_{g,n}(\x)$, and
therefore its coefficients $(c_{g,n}(\alpha))_\alpha$, can be computed using a topological recursion formula
proved by Mirzakhani in \cite{mirzakhani2007}.

More precisely, the coefficients $(c_{g,n}(\alpha))_{\alpha}$ of the volume $V_{g,n}(\x)$ can be expressed as a linear
combination of the coefficients of certain volumes $V_{g',n'}(\x)$, with $n' \geq 1$ and $|\chi'| = 2g'-2+n'$ is strictly smaller than $|\chi| = 2g-2+n$. This ultimately allows the computation of all volume
polynomials $V_{g,n}(\x)$ with non-zero $n$, starting only with the expressions for the volumes $V_{g,n}(\x)$ when
$\abso{\chi} = 1$, which are already known.

\paragraph{Topological enumeration}

In order to state the recursion formula, and the numerous terms it contains, let us first sketch out its
topological interpretation. We consider a bordered hyperbolic surface $X \in \mathcal{M}_{g,n}(\x)$. Our
objective is to `construct' $X$ using smaller pieces.  One way to do so is the following. We focus on one
boundary component of $X$: the first one, $b_1$, for instance. We will try to remove  a
pair of pants containing the boundary component $b_1$ from the surface $X$. Since the Euler characteristic of the pair of pants
is $-1$, the Euler characteristic obtained after removing the pair of pants will decrease in absolute value.

There are many topological types of embedded pairs of pants bounded by~$b_1$. They can be
arranged in three categories.
\begin{enumerate}
\item[(A)] Pairs of pants with two boundary components from $\partial X$, the component $b_1$
  and $b_j$ for a $j \in \{2, \ldots, n\}$. Then, the signature of the surface obtained when
  removing this pair of pants is $(g,n-1)$, with $n-1 \geq 1$.
  
  \begin{center}
    \includegraphics[scale=0.22]{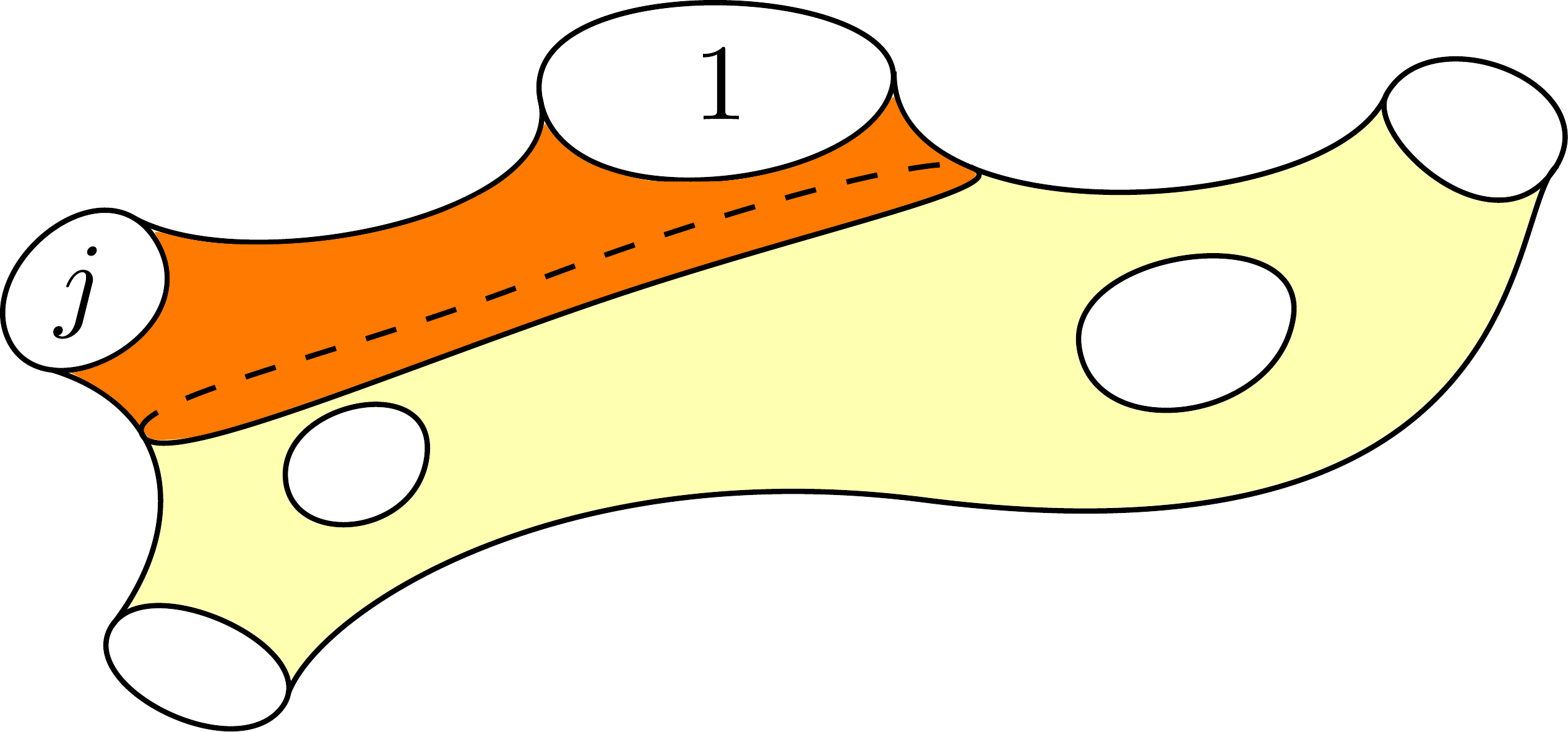}  \hspace{0.5cm}
    \includegraphics[scale=0.9]{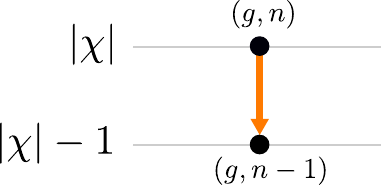}
  \end{center}
\item[(B)] \emph{Non-separating} pairs of pants, that is to say pairs of pants delimited by the
  boundary component $b_1$ and two inner curves, and such that the surface obtained when
  removing the pair of pants is still connected. The signature of the complement is then always equal to $(g-1,n+1)$.

  \begin{center}
      \includegraphics[scale=0.22]{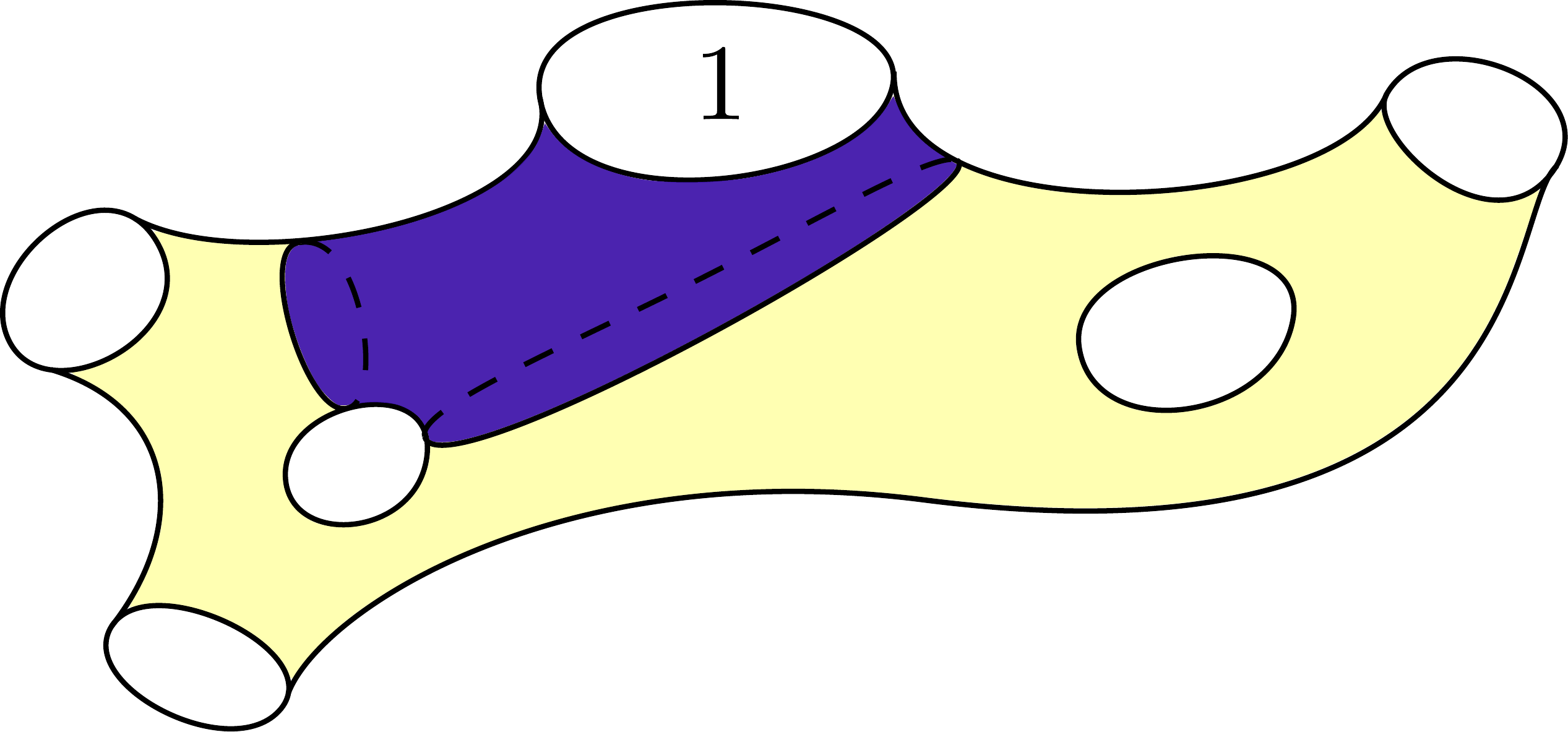}  \hspace{0.5cm}
  \includegraphics[scale=0.9]{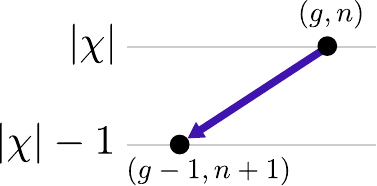}
  \end{center}
\item[(C)] \emph{Separating} pairs of pants, that is to say pairs of pants delimited by the boundary
  component $b_1$ and two inner curves, and which separate the surface into two connected
  components. The topological situation can then be entirely described by the genus $g'$ of one of the
  components (the other genus being $g-g'$), and a partition $(I,J)$ of the boundary components
  $\{2, \ldots, n\}$ of $X$. Note that the only cases which will appear are those for which
  $2 g' - 2 + |I|+1 > 0$ and $2 (g-g') - 2 + |J|+1 > 0$. Let $\mathcal{I}_{g,n}$ denote the set of
  all these topological possibilities.
  
  \begin{center}
    \includegraphics[scale=0.22]{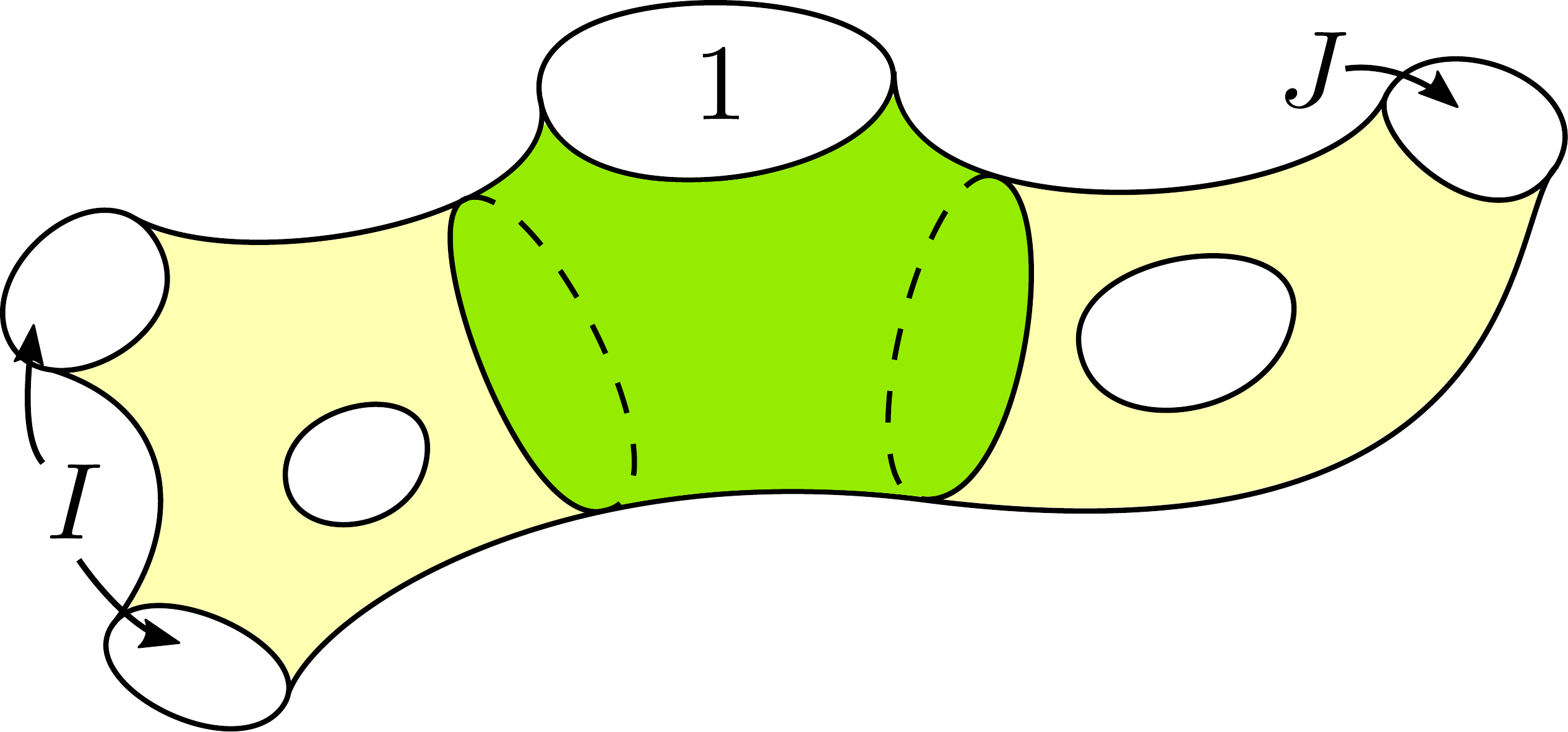} \hspace{0.5cm}
    \includegraphics[scale=0.9]{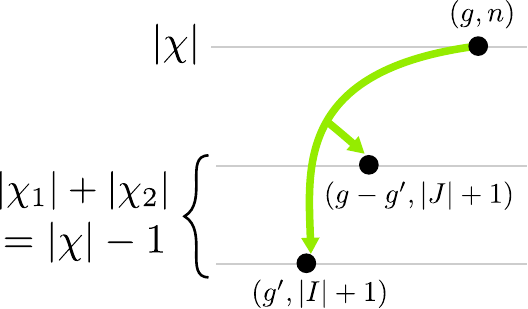}
  \end{center}
\end{enumerate}

\paragraph{The formula}

The coefficients of the volume $V_{g,n}(\x)$ can be expressed as a linear combination of the
coefficients of all the embedded surfaces we encountered in this enumeration.

\begin{thm}[{\cite{mirzakhani2007}}]
  \label{theo:rec_mirz}
  The coefficients of the volume polynomial $V_{g,n}(\x)$ can be written as a sum of three
  contribution, corresponding to the cases (A-C):
  \begin{equation}
    \label{e:recurrence_decomp}
    c_{g,n}(\alpha)
    = \sum_{j=2}^n \mathcal{A}_{g,n}^{(j)}(\alpha) + \mathcal{B}_{g,n}(\alpha) + \sum_{\iota \in \mathcal{I}_{g,n}} \mathcal{C}_{g,n}^{(\iota)}(\alpha).
  \end{equation}
  Each of these terms is a combination of coefficients of the volumes of the corresponding embedded
  surfaces:
  \begin{align}
    \label{e:rec_other_bound}
    \mathcal{A}_{g,n}^{(j)}(\alpha)
    & = 8 \; (2\alpha_j+1) \sum_{i=0}^{+\infty}  
      u_{i} \; c_{g,n-1}(i+\alpha_1+\alpha_j-1,\alpha_2,\ldots,\hat{\alpha}_j,\ldots,\alpha_n) \\
    \label{e:rec_non_sep}
    \mathcal{B}_{g,n}(\alpha)
    & =  16  \sum_{i=0}^{+\infty} \sum_{k_1+k_2=i+\alpha_1-2}
      u_i \; c_{g-1,n+1}(k_1,k_2,\alpha_2,\ldots,\alpha_n)\\
    \label{e:rec_sep}
    \mathcal{C}_{g,n}^{(\iota)}(\alpha)
    & =  16 \sum_{i=0}^{+\infty} \sum_{k_1+k_2=i+\alpha_1-2}
      u_{i} \;  c_{g',|I|+1}(k_1,\alpha_{I}) \; c_{g-g',|J|+1}(k_2, \alpha_{J}),
  \end{align}
  where for any $i \geq 0$, 
  \begin{equation*}
    u_i = \left\{
      \begin{tabular}{ll}
        $\zeta(2i) (1-2^{1-2i})$ & when $i>0$ \\
        $\frac{1}{2}$ & when $i=0$.
%        $0$ & when $i<0$.
      \end{tabular}
    \right.
  \end{equation*} 
\end{thm}

Note that all of the sums in the previous statement are finite because a coefficient $c_{g',n'}(\beta)$ is always equal to zero
if $|\beta| > 3g'-3+n'$, and therefore non-zero terms always satisfy $i \leq 3g-3+n-|\alpha|$.

\begin{expl}
  The coefficients that intervene when computing $V_{g,n}(\x)$ for each $(g,n)$ such that
  $|\chi| \leq 3$ are represented by the arrows in Figure~\ref{fig:recursion_graph_32}.

\begin{figure}[h]
  \centering \includegraphics[scale=1.3]{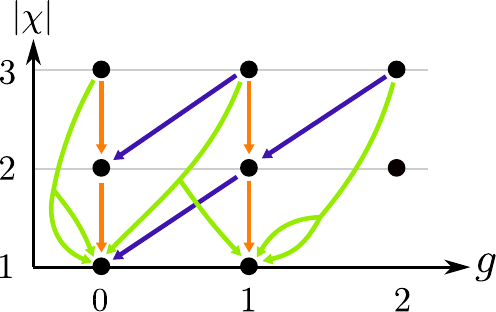}
  \caption{Dependency of the coefficients of the volume polynomials $V_{g,n}(\x)$ when
    $|\chi|=2g-2+n \leq 3$. Note that all the coefficients for which $n \neq 0$ can therefore be
    computed thanks to the coefficients for which $|\chi| = 1$.}
  \label{fig:recursion_graph_32}
\end{figure}
\end{expl}

\paragraph{Sequence $(u_i)_i$ and first properties}

In order to use the topological recursion formula stated in \cref{theo:rec_mirz}, we need some
information of the sequence $(u_i)_i$ that appears in it.

\begin{lem}[{\cite[Lemma 3.1]{mirzakhani2013}}]
  \label{lemm:behaviour_u_i}
  The sequence $(u_i)_i$ is increasing, converges to~$1$ as $i$ approaches
  infinity, and there exists a constant $C>0$ such that
  \begin{equation}
    \label{e:speed_convergence_u}
    \forall i, \quad 0 \leq u_{i+1} - u_i \leq \frac{C}{4^i} \cdot 
  \end{equation}
\end{lem}

We can deduce from the monotonicity of the sequence $(u_i)_i$ the fact that the coefficients
$(c_{g,n}(\alpha))_\alpha$ are decreasing functions of $\alpha$ in the following sense.
\begin{lem}
  \label{cor:c_d_decrease}
  We define the following partial order on multi-indices:
  \begin{equation*}
    \alpha \leq \tilde \alpha \quad
    \Leftrightarrow \quad \forall j \in \{1, \ldots, n\}, \alpha_j \leq \tilde \alpha_j.
  \end{equation*}
  Then, the coefficients $(c_{g,n}(\alpha))_\alpha$ decrease with the
  multi-index $\alpha \in \N_{0}^n$.  In particular,
  \begin{equation}
    \label{eq:c_alpha_Vgn}
    \forall \alpha \in \N_0^n, \quad
    0 \leq c_{g,n}(\alpha) \leq V_{g,n}.
  \end{equation}
\end{lem}

\begin{proof}
  By symmetry of the coefficients, we can reduce the problem to proving that for any multi-indices
  $\alpha$ and $\tilde \alpha = (\tilde \alpha_1, \alpha_2, \ldots, \alpha_n)$ such that
  $\tilde \alpha_1 \geq \alpha_1$, $c_{g,n}(\tilde \alpha) \leq c_{g,n}(\alpha)$.  More precisely,
  we will show that every single term in \cref{e:recurrence_decomp} is smaller for the index
  $\tilde \alpha$ than it is for $\alpha$. The method being the same for every contribution, so we
  only detail the proof of the fact that
  $\mathcal{B}_{g,n}(\tilde \alpha) \leq \mathcal{B}_{g,n}(\alpha)$. By \cref{e:rec_non_sep}, if we use the convention
  $u_i = 0$ for $i<0$,
  \begin{multline*}
    \mathcal{B}_{g,n}(\alpha)
    - \mathcal{B}_{g,n}(\tilde \alpha) \\
    =  16 \sum_{k_1,k_2 \geq 0}
    (\underbrace{u_{k_1+k_2+2-\alpha_1} - u_{k_1+k_2+2-\tilde \alpha_1}}_{\geq 0}) \;
    c_{g-1,n+1}(k_1,k_2,\alpha_2,\ldots,\alpha_n) \geq 0.
  \end{multline*}
\end{proof}

\subsection{Estimates of ratios of Weil--Petersson volumes}
\label{sec:first-estim-rati}

Let us now review known estimates on ratios of Weil--Petersson volumes in the large-genus limit. These
properties have been established in \cite{mirzakhani2013} using several recursion formulas for
Weil--Petersson volumes \cite{mirzakhani2007,do2009,liu2009}, amongst which the one presented in \cref{s:recu_form}.

\paragraph{Same Euler characteristic}

Since two surfaces with the same Euler characteristic are at the same height in the recursion
formula, one could expect the volumes $V_{g,n}$ and $V_{g-1,n+2}$ to be of similar size.  This
is indeed the case: by \cite[Theorem 3.5]{mirzakhani2013}, for all $n \geq 0$, there is a constant
$C_n >0$ such that for any integer $g\geq 0$ satisfying $2g-2+n>0$,
\begin{equation}
  \label{e:volume_same_euler}
  \abso*{\frac{V_{g-1,n+2}}{V_{g,n}} - 1} \leq \frac{C_n}{\jap{g}} \cdot
\end{equation}

\paragraph{Adding a cusp}

We can furthermore compare $V_{g,n}$ and $V_{g,n+1}$ using \cite[Lemma 3.2]{mirzakhani2013}: for any
$g, n \geq 0$ such that $2g-2+n>0$,
\begin{equation}
  \label{e:volume_n_plus_one}
  \frac{1}{12} \paren*{1 - \frac{\pi^2}{10}}
  < \frac{(2g-2+n)V_{g,n}}{V_{g,n+1}} < \frac{\pi \cosh(\pi) - \sinh(\pi)}{2\pi^2} \cdot
\end{equation}
The fact that $V_{g,n+1}$ grows roughly like $(2g-2+n) V_{g,n}$ can be interpreted the following
way: in order to sample a surface of signature $(g,n+1)$, we can start by sampling a surface of
signature $(g,n)$. We then need to decide where to add a cusp, by picking a point on the surface of
area proportional to $2g-2+n$.

\paragraph{Cutting into two smaller surfaces}

Since we can cut surfaces of signature $(g,n)$ into two surfaces of respective signatures $(g_1,n_1+1)$ and
$(g_2,n_2+1)$ with $g_1+g_2=g$ and $n_1+n_2=n$, one could expect the product
$V_{g_1,n_1+1} \times V_{g_2,n_2+1}$ to be of similar size as $V_{g,n}$. Actually, these quantities are much
smaller. Indeed, by \cite[Lemma 3.3]{mirzakhani2013}, for any $n \geq 0$, there exists a constant $C_n>0$
satisfying the following. For any integer $g \geq 0$ such that $2g-2+n>0$ and any integers $n_1$, $n_2$ such that
$n_1+n_2=n$,
\begin{equation}
  \label{e:volume_sum}
  \sum_{\substack{g_1+g_2=g \\ 2g_i+n_i>1}} V_{g_1,n_1+1} V_{g_2,n_2+1} \leq C_n \frac{V_{g,n}}{\jap{g}} \cdot
\end{equation}
The presence of this decay in $1/\jap{g}$ is linked to the fact that typical surfaces of large genus
are very well-connected, and therefore quite difficult to cut into smaller pieces -- a concrete manifestation of this phenomenon can be found in the
comparison of Theorem
4.2 and Theorem 4.4 in \cite{mirzakhani2013}.

\paragraph{Cutting into more surfaces}

In this article, we will need a new version of \cref{e:volume_sum} with additional powers of the
genus.

\begin{lem} 
  \label{lem:volume_sum_K}
  Let $n, N_1, N_2 \geq 0$ be integers. There exists a constant $C_{n,N_1,N_2}$ satisfying the following. For any
  integer $g \geq 0$
  such that $2g-2+n>0$ and any integers $n_1$, $n_2$ such that $n_1+n_2=n$,
  \begin{equation}
    \label{eq:volume_sum_K}
    \sum_{\substack{g_1+g_2=g \\ 2g_i+n_i > N_i+1}}
    \frac{V_{g_1,n_1+1} V_{g_2,n_2+1}}{\jap{g_1}^{N_1}\jap{g_2}^{N_2}}
    \leq C_{n,N_1,N_2} \frac{V_{g,n}}{\jap{g}^{N_1+N_2+1}} \cdot
  \end{equation}
 \end{lem}

 We draw the reader's attention to the fact that the sum is only taken over the set of indices $(g_i, n_i)$ such that
 $2g_i+n_i> N_i+1$. As we will see in the following proof, this is necessary and the result is false if we add a term
 with $1<2g_i+n_i \leq N_i+1$.
 
\begin{proof}
  The proof is an induction on the integer $N_1+N_2$, the case $N_1=N_2=0$ corresponding to \cref{e:volume_sum}. 

  Let $N_1, N_2 \geq 0$ such that $N := N_1+N_2 >0$. We assume the property at the rank $N-1$. By symmetry, we
  can assume that $N_1 \geq N_2$, and in particular $N_1>0$. Then, for any $n_1$, $n_2$ such that
  $n_1+n_2=n$, the left hand side of \cref{eq:volume_sum_K} restricted to the terms where $g_1>0$
  (which only exist if $g>0$) satisfies
  \begin{align*}
    \sum_{\substack{g_1+g_2=g \\ 2g_i+n_i > N_i+1 \\g_1>0}}
    \frac{V_{g_1,n_1+1} V_{g_2,n_2+1}}{\jap{g_1}^{N_1}\jap{g_2}^{N_2}}
    = \mathcal{O}_{n_1} \Bigg( \sum_{\substack{g_1' + g_2 = g-1 \\ 2g_1' + n_1' > N_1 \\ 2g_2+n_2 > N_2+1}}
    \frac{V_{g_1',n_1'+1} V_{g_2,n_2+1}}{\jap{g_1}^{N_1-1}\jap{g_2}^{N_2}} \Bigg)
  \end{align*}
  since $V_{g_1,n_1+1} /\jap{g_1} = \mathcal{O}_{n_1}(V_{g_1-1,n_1+2})$ by
  \cref{e:volume_same_euler,e:volume_n_plus_one}, and thanks to the change of indices $g_1'=g_1-1$,
  $n_1'=n_1+1$.  By the induction hypothesis, this sum is
  \begin{equation*}
    \mathcal{O}_{n+1,N_1-1,N_2}\paren*{\frac{V_{g-1,n+1}}{\jap{g-1}^{N}}}
    = \mathcal{O}_{n,N_1,N_2} \paren*{\frac{V_{g,n}}{\jap{g}^{N+1}}}
  \end{equation*}
  by \cref{e:volume_same_euler,e:volume_n_plus_one} again.

  As a consequence, we are left to bound the term for which $g_1=0$. If such a term is present in the sum,
  then the integer $n_1 = 2g_1+n_1$ satisfies $n_1 > N_1+1$, and hence $n - n_2 -1 \geq N_1+1$. Then, the term
  of the sum is
  \begin{equation*}
    \frac{V_{0,n_1+1} V_{g,n_2+1}}{\jap{0}^{N_1} \jap{g}^{N_2}}
    = \mathcal{O}_{n} \paren*{\frac{V_{g,n}}{\jap{g}^{N_2+n-n_2-1}}}
    = \mathcal{O}_{n} \paren*{\frac{V_{g,n}}{\jap{g}^{N+1}}} 
  \end{equation*}
  by \cref{e:volume_n_plus_one} applied $n-n_2-1$ times.
\end{proof}

\subsection{The leading term of the asymptotic expansion}

Let us conclude this preliminary section by a detailed proof of the following first-order estimate. This will allow us
to present a few ideas that will be used in the general case.

\begin{prp}[{\cite[Proposition 3.1]{mirzakhani2019}}]
  \label{prop:main_term_vgn}
  For any $n \geq 1$, $g \geq 0$ such that $2g-2+n>0$,  and any length vector $\x \in \R_{\geq 0}^n$,
  \begin{equation*}
    \frac{V_{g,n}(\x)}{V_{g,n}} = \prod_{j=1}^n \sinhc \div{x_j}
    +\O[n]{\frac{\abso{\x}}{\jap{g}} \exp \div{x_1+ \ldots + x_n}}.
  \end{equation*}
\end{prp}

This proposition comes as a consequence of the expression for the volume
polynomials in terms of their coefficients $(c_{g,n}(\alpha))_\alpha$,
together with the following first-order estimate for the coefficients.

\begin{lem}
  \label{prop:main_term_calpha}
  For any $n \geq 1$, $g \geq 0$ such that $2g-2+n>0$, and any multi-index $\alpha \in \N_0^n$,
  \begin{equation*}
    c_{g,n}(\alpha) = V_{g,n} + \O[n]{\abso{\alpha}^2 \frac{V_{g,n}}{\jap{g}}}.
  \end{equation*}
\end{lem}

\begin{rem}
  We insist on the fact that this estimate is true \emph{for any $\alpha$} and not only for
  multi-indices $\alpha$ such that $\abso{\alpha} \leq 3g-3+n$. Indeed, if $\abso{\alpha} > 3g-3+n$,
  then the bound is trivial, because $c_{g,n}(\alpha)=0$ and $\frac{\abso{\alpha}^2}{\jap{g}} \gg 1$.
\end{rem}

We first prove \cref{prop:main_term_calpha} in \cref{sec:discrete_der_zero,sec:discrete_int_zero}, and then
deduce \cref{prop:main_term_vgn} from it in \cref{sec:coeff_to_vol_zero}.

\subsubsection{First-order estimate of the discrete derivative}
\label{sec:discrete_der_zero}

\Cref{prop:main_term_calpha} states that the coefficients $c_{g,n}(\alpha)$ are almost constant, equal to the value
$V_{g,n} = c_{g,n}(\z{n})$. We will prove this by estimating the \emph{discrete derivatives} of the coefficients $c_{g,n}(\alpha)$.

\begin{lem}
  \label{lemm:main_term_calpha}
  For any integers $g \geq 0$ and $n \geq 1$ satisfying $2g-2+n>0$, and any multi-index $\alpha \in \N_0^n$,
  \begin{equation*}
    \delta_1 c_{g,n}(\alpha) = \O[n]{\jap{\alpha} \frac{V_{g,n}}{\jap{g}}}.
  \end{equation*}
\end{lem}

Note that, by symmetry of the volume coefficients, this result is also true if we replace $\delta_1$ by $\delta_i$ for
any $i \in \{1, \ldots, n\}$. 

\begin{proof}
  The result is trivially true when $\abso{\chi} = 2g-2+n =1$, so we can assume that it is not the case and apply
  Mirzakhani's topological recursion formula, \cref{theo:rec_mirz}:
  \begin{equation*}
    \delta_1c_{g,n}(\alpha)
    = \sum_{j=2}^n \delta_1\mathcal{A}_{g,n}^{(j)}(\alpha)
    + \delta_1\mathcal{B}_{g,n}(\alpha)
    + \sum_{\iota \in \mathcal{I}_{g,n}} \delta_1\mathcal{C}_{g,n}^{(\iota)}(\alpha).
  \end{equation*}
  We prove that each of these three terms is $\O[n]{\jap{\alpha} V_{g,n}/\jap{g}}$ separately thanks to
  their respective expressions, \cref{e:rec_other_bound,e:rec_non_sep,e:rec_sep}.

  Let us begin by the first sum. For a $j \geq 2$, we write \cref{e:rec_other_bound} for
  $\mathcal{A}_{g,n}^{(j)}(\alpha)$ and
  $\mathcal{A}_{g,n}^{(j)}(\alpha_1+1, \alpha_2, \ldots, \alpha_n)$, isolating the term $i=0$ in the
  first sum and using a change of index on the sum over $i \geq 1$. We obtain
  \begin{align*}
    \delta_1\mathcal{A}_{g,n}^{(j)}(\alpha)
    =& \; 4 \; (2\alpha_j+1) \;
       c_{g,n-1}(\alpha_1+\alpha_j-1,\alpha_2, \ldots, \hat{\alpha}_j, \ldots, \alpha_n) \\
     & + 8 \; (2\alpha_j+1) \sum_{i=0}^{+\infty} (u_{i+1} - u_{i}) \;
       c_{g,n-1}(i+\alpha_1+\alpha_j,\alpha_2, \ldots, \hat{\alpha}_j, \ldots, \alpha_n).
  \end{align*}
  But we know by \cref{cor:c_d_decrease} that for any multi-index $\beta \in \N_0^{n-1}$,
  \begin{equation*}
    0 \leq c_{g,n-1}(\beta) \leq V_{g, n-1}.
  \end{equation*}
  Then,
  \begin{equation*}
    0 \leq \delta_1\mathcal{A}_{g,n}^{(j)}(\alpha)
    \leq 8(2 \alpha_j+1) V_{g,n-1}
    = \O[n]{\jap{\alpha_j} \frac{V_{g,n}}{\jap{g}}}
  \end{equation*}
  because $\sum_{i=0}^{+ \infty} (u_{i+1} - u_i) = \lim u - u_0 = 1 - \frac 12 = \frac 12$ by
  \cref{lemm:behaviour_u_i}, and thanks to \cref{e:volume_n_plus_one}. Since there are $n-1 = \O[n]{1}$
  possible values for $j$,
  \begin{equation}
    \label{e:estimate_Aj_one}
    \sum_{j=2}^n \delta_1 \mathcal{A}_{g,n}^{(j)}(\alpha) = \O[n]{\jap{\alpha}\frac{V_{g,n}}{\jap{g}}}.
  \end{equation}

  We now look at the non-separating term $\delta_1\mathcal{B}_{g,n}(\alpha)$. Note that this term
  only appears whenever $g \geq 1$. By the same method, this time applied to \cref{e:rec_non_sep},
  \begin{align*}
    \delta_1\mathcal{B}_{g,n}(\alpha)
    = & \, 8 \sum_{k_1+k_2=\alpha_1-2} c_{g-1,n+1}(k_1,k_2,\alpha_2, \ldots, \alpha_n) \\
    & + 16 \sum_{i=0}^{+\infty} \sum_{k_1+k_2=i+\alpha_1-1}
    (u_{i+1} - u_i) \; c_{g-1,n+1}(k_1,k_2,\alpha_2, \ldots, \alpha_n)\cdot
  \end{align*}
  By \cref{cor:c_d_decrease}, for any multi-index $\beta \in \N_0^{n+1}$,
  \begin{equation*}
    0 \leq c_{g-1,n+1}(\beta) \leq V_{g-1, n+1} = \O[n]{\frac{V_{g,n}}{\jap{g}}}
  \end{equation*}
  thanks to \cref{e:volume_n_plus_one,e:volume_same_euler}. Then, 
  \begin{equation}
    \label{e:estimate_B_one}
    \delta_1\mathcal{B}_{g,n}(\alpha)
    = \O[n]{\alpha_1 \frac{V_{g,n}}{\jap{g}} +  \sum_{i=0}^{+\infty} (i+\alpha_1)
      (u_{i+1} - u_i) \frac{V_{g,n}}{\jap{g}}}
    = \O[n]{\jap{\alpha_1}\frac{V_{g,n}}{\jap{g}}}
  \end{equation}
  because the series $\sum_i (u_{i+1} - u_i)$ and $\sum_i i(u_{i+1} - u_i)$ converge.

  Finally, for any configuration $\iota = (g', I, J) \in \mathcal{I}_{g,n}$,
  \begin{align*}
    \delta_1 \mathcal{C}_{g,n}^{(\iota)}(\alpha) 
    = & \, 8 \sum_{k_1+k_2=\alpha_1-2}
        c_{g',|I|+1}(k_1,\alpha_{I}) \; c_{g-g',|J|+1}(k_2, \alpha_{J}) \\
      & + 16 \sum_{i=0}^{+\infty} \sum_{k_1+k_2=i+\alpha_1-1}
        (u_{i+1} - u_{i}) \;  c_{g',|I|+1}(k_1,\alpha_{I}) \; c_{g-g',|J|+1}(k_2, \alpha_{J}) \\
    = & \, \O[n]{\jap{\alpha_1} V_{g',|I|+1} V_{g-g',|J|+1}}. 
  \end{align*}
  As a consequence, 
  \begin{equation*}
    \sum_{\iota \in \mathcal{I}_{g,n}} \delta_1 \mathcal{C}_{g,n}^{(\iota)} (\alpha)= 
    \mathcal{O}_n \Bigg( \jap{\alpha_1} \sum_{\substack{g_1+g_2=g \\ n_1+n_2 = n-1 \\ 2g_i+n_i>1}} V_{g_1,n_1+1}
    V_{g_2,n_2+1} \Bigg)
  \end{equation*}
  and therefore, by \cref{e:volume_sum,e:volume_n_plus_one},
  \begin{equation}
    \label{e:estimate_C_one}
    \sum_{\iota \in \mathcal{I}_{g,n}} \delta_1 \mathcal{C}_{g,n}^{(\iota)}(\alpha)
   = \O[n]{\jap{\alpha_1} \frac{V_{g,n-1}}{\jap{g}}}
    = \O[n]{\jap{\alpha_1} \frac{V_{g,n}}{\jap{g}^2}}.
  \end{equation}
  The conclusion follows from adding \cref{e:estimate_Aj_one,e:estimate_B_one,e:estimate_C_one}.
\end{proof}

\subsubsection{A discrete integration formula}
\label{sec:discrete_int_zero}

In order to go from an estimate of discrete derivatives to an estimate on actual
coefficients, we use the following discrete integration lemma.

\begin{lem}
  \label{lemm:discrete_integral}
  Let $n \geq 1$ be an integer. 
  For any $v : \N_0^n \rightarrow \R$,
  \begin{equation*}
    v(\alpha)
    = v(\z{n}) - \sum_{i=1}^n \sum_{k=0}^{\alpha_i-1} \delta_i v(\z{i-1}, k, \alpha_{i+1}, \ldots, \alpha_n).
  \end{equation*}
\end{lem}

\Cref{prop:main_term_calpha} then directly follows from this formula and our first-order estimate on the discrete
derivatives, \cref{lemm:main_term_calpha}.

\begin{proof}[Proof of \cref{lemm:discrete_integral}]
  We observe that for any index $i$, the sum over $k$ is a telescopic sum:
  \begin{align*}
    S_i & := \sum_{k=0}^{\alpha_i-1} \delta_i v(\z{i-1}, k, \alpha_{i+1}, \ldots, \alpha_n)\\
    & = \sum_{k=0}^{\alpha_i-1}
    \brac*{v(\z{i-1}, k, \alpha_{i+1}, \ldots, \alpha_n)
      -  v(\z{i-1}, k+1, \alpha_{i+1}, \ldots, \alpha_n)}\\
    & =  v(\z{i}, \alpha_{i+1}, \ldots, \alpha_n)
      -  v(\z{i-1}, \alpha_i, \alpha_{i+1}, \ldots, \alpha_n).
  \end{align*}
  As a consequence, $\sum_{i=1}^n S_i = v(\z{n}) - v(\alpha)$, which what was claimed.
\end{proof}

\subsubsection{From the coefficient estimate to the volume estimate}
\label{sec:coeff_to_vol_zero}

Let us finally prove that \cref{prop:main_term_calpha} implies \cref{prop:main_term_vgn}.

\begin{proof}
  Using the expression of $\sinhc$ as a power series, we can write
  \begin{equation*}
    V_{g,n}(\x)
    - V_{g,n} \prod_{j=1}^n \sinhc \div{x_j}
    = \sum_{\substack{\alpha \in \N_0^n \\ \alpha \neq \z{n}}} (c_{g,n}(\alpha) - V_{g,n}) \prod_{j=1}^n \frac{x_j^{2\alpha_j}}{2^{2 \alpha_j}
      (2 \alpha_j + 1)!} \cdot
  \end{equation*}
  As a consequence, by the triangle inequality and \cref{prop:main_term_calpha},
  \begin{equation}
        \label{eq:estimate_vgn_main_term}
    \abso*{V_{g,n}(\x)
      - V_{g,n} \prod_{j=1}^n \sinhc \div{x_j}}
    \leq C_n \frac{V_{g,n}}{\jap{g}} \sum_{\substack{\alpha \in \N_0^n \\ \alpha \neq \z{n}}}
    \abso{\alpha}_\infty^2 \prod_{j=1}^n \frac{x_j^{2\alpha_j}}{2^{2 \alpha_j}
      (2 \alpha_j + 1)!} \cdot
  \end{equation}
  We cut the sum over $\alpha$ in \cref{eq:estimate_vgn_main_term} depending on the index $j$ for
  which $\abso{\alpha}_\infty = \alpha_j$.  Since $\alpha_j^2 \leq (2\alpha_j+1)(2\alpha_j)/4$,
  \begin{equation*}
    \sum_{\alpha_j = 1}^{+\infty} \frac{\alpha_j^2 \, x_j^{2\alpha_j}}{2^{2 \alpha_j} (2 \alpha_j + 1)!}
    \leq \sum_{k = 0}^{+\infty} \frac{x_j^{2k+2}}{2^{2 k} (2k + 1)!}
    = 2 x_j \sinh \div{x_j} \leq \abso{\x}  \exp \div{x_j}.
  \end{equation*}
  Also, for any $i$, 
  \begin{equation*}
    \sum_{\alpha_i = 0}^{+\infty} \frac{x_i^{2\alpha_i}}{2^{2 \alpha_i} (2 \alpha_i + 1)!}
    \leq \sum_{k = 0}^{+\infty} \frac{x_i^{k}}{2^{k} k!}
    \leq \exp \div{x_i}.
  \end{equation*}
  This allows us to conclude.
\end{proof}

\section{An explicit second-order expansion}
\label{s:second_order}

We now possess all the tools that are required to compute the second term of the asymptotic expansion of
$V_{g,n}(\x)$. We recall the notation $\cd{x} = \cosh \div{x}$ and $\sd{x} = \sinhc \div{x}$.  Let us prove the
following statement.
\begin{thm}
  \label{prop:vgn_order_two}
   For any integers $g \geq 0$ and $n \geq 1$
  such that $2g-2+n>0$, and any $\x \in \R_{\geq 0}^n$,
  \begin{equation*}
    \frac{V_{g,n}(\x)}{V_{g,n}}
    = F_{g,n}^{(1)}(\x) 
    + \O[n]{\frac{\jap{\x}^3}{\jap{g}^2} \; \exp \div{x_1+ \ldots + x_n}}
  \end{equation*}
  where the functions $F_{g,n}^{(1)}$ is the function defined by:
  \begin{align*}
     F_{g,n}^{(1)}(\x) 
     = & \prod_{1\leq k \leq n} \sd{x_k} \\ 
    & + 8 \; \frac{V_{g-1,n+1} \1{g \geq 1}}{V_{g,n}}
      \sum_{i=1}^n \brac*{\cd{x_i} +1
      - \paren*{\frac{x_i^2}{16} + 2} \sd{x_i}}
      \prod_{k\neq i} \sd{x_k}  \\
    & - 4 \; \frac{V_{g,n-1}}{V_{g,n}} \sum_{1 \leq i < j \leq n}
      \brac*{\cd{x_i} \, \cd{x_j} +1 - 2 \, \sd{x_i} \, \sd{x_j}}
      \prod_{k \notin \{ i, j\}} \sd{x_k}.
  \end{align*}
\end{thm}

\Cref{thm:order_two} can then be obtained by using the expansions proved in~\cite{mirzakhani2015}: for any $g
\geq 1$,
  \begin{align*}
    \frac{V_{g,n-1}}{V_{g,n}}
    & = \frac{1}{8 \pi^2g} + \O[n]{\frac{1}{g^2}}\\
    \frac{V_{g-1,n+1}}{V_{g,n}}
    & = \frac{V_{g-1,n+1}}{V_{g,n-1}} \frac{V_{g,n-1}}{V_{g,n}}
      = \frac{1}{8 \pi^2g} + \O[n]{\frac{1}{g^2}}.
  \end{align*}

The key ingredient in the proof of \cref{prop:vgn_order_two} is the following approximation result for the volume
coefficients $(c_{g,n}(\alpha))_\alpha$ up to errors of size $V_{g,n}/\jap{g}^2$.

\begin{prp}
  \label{prop:cgn_order_two}
  For any  integers $g \geq 0$ and $n \geq 1$ such that $2g-2+n>0$,  
  \begin{equation*}
    \forall \alpha \in \N_0^n, \quad
    c_{g,n}(\alpha) = \hat{c}_{g,n}^{(1)}(\alpha) + \O[n]{\abso{\alpha}^4 \frac{V_{g,n}}{\jap{g}^2}}
  \end{equation*}
  where $\hat{c}_{g,n}^{(1)} : \N_0^n \rightarrow \R$ is the function defined by:
  \begin{align*}
    \hat{c}_{g,n}^{(1)}(\alpha)
    = V_{g,n}
    &+ 8 V_{g-1,n+1} \1{g \geq 1} \sum_{i=1}^n \paren*{p_1(\alpha_i) + \1{\alpha_i=0} - \frac{p_2(\alpha_i)}{4} - 2} \\
    & - 4 V_{g,n-1} \1{n \geq 2} \sum_{1 \leq i < j \leq n} 
      \paren*{p_1(\alpha_i) \, p_1(\alpha_j) + \1{\alpha_i = \alpha_j=0} - 2}
  \end{align*}
  and $p_1(X) := 2X+1$, $p_2(X) := (2X+1)(2X)$.
\end{prp}

Similarly to the first-order case presented before, the proof of \cref{prop:cgn_order_two} spans over
\cref{sec:discrete-derivative,sec:discrete-integration}, and we then deduce \cref{prop:vgn_order_two} from it in
\cref{s:proof_volume_order_2}.

\subsection{Second-order estimate of the discrete derivative}
\label{sec:discrete-derivative}

In order to expand the coefficients $(c_{g,n}(\alpha))_{\alpha \in \N_0^n}$, we first estimate the
discrete derivative $\delta_1 c_{g,n}(\alpha)$. 

\begin{lem}
  \label{lemm:second_order_derivative}
  For any  integers $g \geq 0$ and $n \geq 1$, satisfying $2g-2+n>0$,  
  \begin{equation*}
    \forall \alpha \in \N_0^n, \quad
    \delta_1 c_{g,n}(\alpha) = \psi_{g,n}^{(1)}(\alpha) + \O[n]{\jap{\alpha}^3\frac{V_{g,n}}{\jap{g}^2}}
  \end{equation*}
  where $\psi_{g,n}^{(1)} : \N_0^n \rightarrow \R$ is the function defined by:
  \begin{align*}
    \psi_{g,n}^{(1)}(\alpha)
    = & 4 \, (4 \alpha_1 - 1 + 2 \, \1{\alpha_1=0}) V_{g-1,n+1} \1{g \geq 1} \\
    & + 4 \sum_{j=2}^n (4 \alpha_j+2 - \1{\alpha_1=\alpha_j=0}) V_{g,n-1} \1{n \geq 2}.
  \end{align*}
\end{lem}

\begin{proof}
  We apply the discrete derivation to Mirzakhani's recursion formula, \cref{theo:rec_mirz}:
  \begin{equation*}
    \delta_1c_{g,n}(\alpha)
    = \sum_{j=2}^n \delta_1\mathcal{A}_{g,n}^{(j)}(\alpha)
    + \delta_1\mathcal{B}_{g,n}(\alpha)
    + \sum_{\iota \in \mathcal{I}_{g,n}} \delta_1\mathcal{C}_{g,n}^{(\iota)}(\alpha).
  \end{equation*}
  We then replace every term by the first-order approximation given by \cref{prop:main_term_calpha}, which will allow us to estimate them
  up to errors of size $V_{g,n}/\jap{g}^2$.

  We notice that the first term is zero if $n=1$. Let us assume otherwise, and take
  $j \in \{2, \ldots, n\}$. As in the first-order case, we can write
  \begin{align*}
    \delta_1\mathcal{A}_{g,n}^{(j)}(\alpha)
    =& \; 4 \, (2\alpha_j+1) \;
       c_{g,n-1}(\alpha_1+\alpha_j-1,\alpha_2, \ldots, \hat{\alpha}_j, \ldots, \alpha_n) \\
    & + 8 \, (2\alpha_j+1) \sum_{i=0}^{+\infty}  
      (u_{i+1} - u_{i}) \;
      c_{g,n-1}(i+\alpha_1+\alpha_j,\alpha_2, \ldots, \hat{\alpha}_j, \ldots, \alpha_n) \\
    =: & \; T_1+T_2.
  \end{align*}
  \begin{itemize}
  \item Estimate of the term $T_1$:
    \begin{itemize}
    \item If $\alpha_1=\alpha_j=0$, then $\alpha_1+\alpha_j-1 <0$ and therefore $T_1=0$.
    \item Otherwise, by \cref{prop:main_term_calpha} applied to the coefficient in $T_1$,
      \begin{align*}
        T_1
        & = 4 \, (2\alpha_j+1) V_{g,n-1} + \O[n]{(2 \alpha_j+1)\abso{\alpha}^2 \frac{V_{g,n-1}}{\jap{g}}} \\
        & = 4 \, (2\alpha_j+1) V_{g,n-1} + \O[n]{\jap{\alpha}^3 \frac{V_{g,n}}{\jap{g}^2}}
      \end{align*}
      by \cref{e:volume_n_plus_one}.
    \end{itemize}
  \item To estimate of the term $T_2$, we replace the volume coefficients appearing in $T_2$ by their first-order
    approximation and obtain:
    \begin{equation*}
      T_2
       = 8 \, (2 \alpha_j+1) \sum_{i=0}^{+\infty}  
        (u_{i+1} - u_{i}) \paren*{V_{g,n-1} + \O[n]{(i + \abso*{\alpha})^2 \frac{V_{g,n-1}}{\jap{g}}}} 
    \end{equation*}
    by \cref{prop:main_term_calpha}. But \cref{lemm:behaviour_u_i} implies that
    $\sum_{i=0}^{+ \infty} (u_{i+1} - u_i) = \frac 12$ and $\sum_{i=0}^{+ \infty} i(u_{i+1} - u_i)$
    converges. Therefore, by \cref{e:volume_n_plus_one} again,
    \begin{equation*}
      T_2
      = 4 \, (2 \alpha_j+1) V_{g,n-1} +  \O[n]{\jap{\alpha}^3 \frac{V_{g,n}}{\jap{g}^2}}.
    \end{equation*}
  \end{itemize}
  As a conclusion, we have proved that
  \begin{equation*}
    \delta_1 \mathcal{A}_{g,n}^{(j)}(\alpha) =
    \begin{cases}
      4 \, V_{g,n-1} + \O[n]{\jap{\alpha}^3 \frac{V_{g,n}}{\jap{g}^2}} & \text{if } \alpha_1 = \alpha_j = 0 \\
      8 \, (2 \alpha_j+1) V_{g,n-1} +  \O[n]{\jap{\alpha}^3 \frac{V_{g,n}}{\jap{g}^2}} & \text{otherwise.}
    \end{cases}
  \end{equation*}
  We rewrite this expression as
  \begin{equation*}
    \delta_1 \mathcal{A}_{g,n}^{(j)}(\alpha) =
    4 \, (4 \alpha_j+2 - \1{\alpha_1=\alpha_j=0}) V_{g,n-1} \1{n \geq 2}
    + \O[n]{\jap{\alpha}^3 \frac{V_{g,n}}{\jap{g}^2}}.
  \end{equation*}

  By the same process, we prove that, when $g \geq 1$, 
  \begin{equation*}
    \delta_1 \mathcal{B}_{g,n}(\alpha) =
    4 \, (4\alpha_1-1 + 2 \, \1{\alpha_1=0}) V_{g-1,n+1} \1{g \geq 1}
    + \O[n]{\jap{\alpha}^3 \frac{V_{g,n}}{\jap{g}^2}}.
  \end{equation*}
  Indeed,
  \begin{align*}
    \delta_1\mathcal{B}_{g,n}(\alpha)
    = & \, 8 \sum_{k_1+k_2=\alpha_1-2} c_{g-1,n+1}(k_1,k_2,\alpha_2, \ldots, \alpha_n) \\
      & + 16 \sum_{i=0}^{+\infty} \sum_{k_1+k_2=i+\alpha_1-1}
        (u_{i+1} - u_i) \; c_{g-1,n+1}(k_1,k_2,\alpha_2, \ldots, \alpha_n) \\
        =: & \,  T_1+T_2.
  \end{align*}
  \begin{itemize}
  \item On the one hand, $T_1$ is equal to zero if $\alpha_1=0$, and otherwise,
    \begin{equation*}
      T_1 = 8 \, (\alpha_1-1) V_{g-1,n+1} + \O[n]{\jap{\alpha}^3 \frac{V_{g,n}}{\jap{g}^2}}
    \end{equation*}
    because $V_{g-1,n+1} = \O[n]{V_{g,n}/\jap{g}}$ by \cref{e:volume_n_plus_one,e:volume_same_euler}.
  \item On the other hand,
    \begin{align*}
      T_2 & = 16  \sum_{i=0}^{+ \infty} (\alpha_1+i)(u_{i+1} - u_i) V_{g-1,n+1}
            + \O[n]{\jap{\alpha}^3 \frac{V_{g-1,n+1}}{\jap{g-1}}}\\
          & = 4 \, (2 \alpha_1+1) V_{g-1,n+1} + \O[n]{\jap{\alpha}^3 \frac{V_{g,n}}{\jap{g}^2}}
    \end{align*}
    because, as before,  $\sum_{i=0}^{+ \infty} (u_{i+1} - u_i) = \frac 12$, and also
    $\sum_{i=0}^{+ \infty} i(u_{i+1} - u_i) = \frac 14$ by \cite[Lemma 2.1]{mirzakhani2015}.
  \end{itemize}
  
  Finally, we observe that, when computing the first order term, we have proved in \cref{e:estimate_C_one} that
  \begin{equation*}
    \sum_{\iota \in \mathcal{I}_{g,n}} \delta_1 \mathcal{C}_{g,n}^{(\iota)}(\alpha)
    = \O[n]{\jap{\alpha} \frac{V_{g,n}}{\jap{g}^2}}.
  \end{equation*}
  As a consequence, the separating term (C) does not contribute to the second-order approximation of $\delta_1c_{g,n}(\alpha)$.

  Summing the different terms $\delta_1 \mathcal{A}_{g,n}^{(j)}$ for $j \in \{2, \ldots, n\}$ and $\delta_1
  \mathcal{B}_{g,n}(\alpha)$ leads to the claim.
\end{proof}

\subsection{Discrete integration of the second-order estimate}
\label{sec:discrete-integration}

We can now prove \cref{prop:cgn_order_two} using \cref{lemm:second_order_derivative} and discrete integration.

\begin{proof}
  By the discrete integration lemma (\cref{lemm:discrete_integral}) and by symmetry of the volume coefficients,
  \begin{equation*}
    c_{g,n}(\alpha)
    = V_{g,n} - \sum_{i=1}^n \sum_{k=0}^{\alpha_i-1} \delta_1 c_{g,n}(k, \z{i-1}, \alpha_{i+1}, \ldots,
    \alpha_n).
  \end{equation*}
  We then apply \cref{lemm:second_order_derivative} to deduce
  \begin{equation*}
    c_{g,n}(\alpha)
    = V_{g,n} - \sum_{i=1}^n \sum_{k=0}^{\alpha_i-1} \psi^{(1)}_{g,n}(k, \z{i-1}, \alpha_{i+1}, \ldots,
      \alpha_n)  + \O[n]{\jap{\alpha}^4 \frac{V_{g,n}}{\jap{g}^2}}.
  \end{equation*}
  This can be rewritten as 
  \begin{equation}
    \label{eq:second_c_t12}
    c_{g,n}(\alpha)
    = V_{g,n} - 4  V_{g-1,n+1} \1{g \geq 1} T_1 - 4 V_{g,n-1} \1{n \geq 2} T_2 + \O[n]{\jap{\alpha}^4 \frac{V_{g,n}}{\jap{g}^2}}
  \end{equation}
  where the quantities $T_1$ and $T_2$ are defined by
  \begin{align*}
    T_1 & := \sum_{i=1}^n \sum_{k=0}^{\alpha_i-1} (4k-1 + 2 \, \1{k=0}) \\
    T_2 & := \sum_{i=1}^n \sum_{k=0}^{\alpha_i-1} \brac*{ (i-1) (2 -
                      \1{k=0}) + \sum_{j=i+1}^n (4 \alpha_j + 2 - \1{k=\alpha_j=0})}.
  \end{align*}
  \begin{itemize}
  \item On the one hand, we observe that for a fixed $i$, the term $\1{k=0}$ contributes to the sum at most once, and
    this occurs if and only if $\alpha_i>0$. Hence, when we perform the sum over $k$, we obtain
    \begin{equation*}
      T_1 = \sum_{i=1}^n \paren*{2 \alpha_i^2 - 3 \alpha_i + 2 \, \1{\alpha_i>0}}.
    \end{equation*}
    We reorder the sum according to the dependency over $\alpha$, and use the fact that
    $\1{\alpha_i>0} = 1 - \1{\alpha_i=0}$, to obtain
        \begin{equation*}
          T_1 =
          \sum_{i=1}^n
          \paren*{ - 2(2\alpha_i+1) - 2 \,\1{\alpha_i=0} + \frac{(2\alpha_i+1)(2 \alpha_i)}{2} + 4}.
        \end{equation*}
  \item On the other hand, by the same method,
    \begin{align*}
      T_2 & = \sum_{i=1}^n \brac*{ (i-1) (2 \alpha_i - \1{\alpha_i>0})
            + \sum_{j=i+1}^n (4 \alpha_i \alpha_j + 2 \alpha_i - \1{\alpha_i>0} \1{\alpha_j=0})} \\
          & = 4 \sum_{i<j} \alpha_i \alpha_j + 2 (n-1) \sum_{i =1}^n \alpha_i
            + \sum_{i>j} (\1{\alpha_i=0} - 1) + \sum_{i<j} (\1{\alpha_i=0} - 1) \1{\alpha_j=0} \\
          & = \sum_{i<j} \paren*{(2\alpha_i+1)(2 \alpha_j+1) + \1{\alpha_i=\alpha_j=0} -2}.
    \end{align*}
    This allows us to conclude, by \cref{eq:second_c_t12}.
  \end{itemize}
\end{proof}

\subsection{From the coefficient estimate to the volume estimate}
\label{s:proof_volume_order_2}

In order to conclude the proof of \cref{prop:vgn_order_two}, we need to compute
\begin{equation*}
  \sum_{\alpha \in \N_0^n} \hat{c}_{g,n}^{(1)}(\alpha) \prod_{j=1}^n \frac{x_j^{2\alpha_j}}{2^{2 \alpha_j} (2 \alpha_j+1)!}
\end{equation*}
where $\hat{c}^{(1)}_{g,n}(\alpha)$ is the approximation of the coefficient $c_{g,n}(\alpha)$ from
\cref{prop:cgn_order_two}.  We have expressed $\hat{c}^{(1)}_{g,n}$ in terms of polynomials $p_1(X) = 2X+1$ and
$p_2(X) = (2X+1)(2X)$ in order to make this computation easier.

Since this will be useful for the general case, let us set some notations.
\begin{nota}
  \label{nota:p_k}
  For any integer $k \geq 0$, we set
  \begin{equation*}
    p_k(X) =  \prod_{j=0}^{k-1} (2X+1-j)
    = (2X+1) (2X) (2X-1) \ldots (2X+2-k),
  \end{equation*}
  with the convention that the empty product is equal to one so that $p_0(X)=1$. 
\end{nota}

Since the polynomials $(p_k)_{k \geq 0}$ are a basis of the set of polynomials, we will be able to express any
polynomial function as a linear combination of these polynomials. The following simple observation is our motivation for
the introduction of these polynomials.

\begin{lem}
  \label{lemm:sum_p_k}
  Let $k \geq 0$ be an integer. For any $x \in \R$,
  \begin{equation*}
    \sum_{\alpha=0}^{+ \infty} \frac{p_{k}(\alpha) \, x^{2 \alpha}}{2^{2 \alpha}(2 \alpha+1)!}
    =
    \begin{cases}
      \frac{x^{k}}{2^{k}} \sinhc \div{x} & \text{if } k \text{ is even} \\
      \frac{x^{k-1}}{2^{k-1}} \cosh \div{x} & \text{if } k \text{ is odd}.
    \end{cases}
  \end{equation*}
\end{lem}

We can now finish the proof of \cref{prop:vgn_order_two}.

\begin{proof}
  By \cref{prop:cgn_order_two} and the expression of $V_{g,n}(\x)$ in terms of $(c_{g,n}(\alpha))_\alpha$,
  \begin{equation*}
    \frac{V_{g,n}(\x)}{V_{g,n}} = F_{g,n}^{(1)}(\x)
    + \O[n]{\frac{V_{g,n}}{\jap{g}^2} \sum_{\substack{\alpha \in \N_0^n \\ \alpha \neq \z{n}}}
      \abso{\alpha}_\infty^4 \prod_{j=1}^n\frac{x_j^{2 \alpha_j}}{2^{2 \alpha_j} (2 \alpha_j+1)!}}
  \end{equation*}
  where
  \begin{equation*}
    F_{g,n}^{(1)}(\x)
    =  \sum_{\alpha \in \N_0^n} \frac{\hat{c}_{g,n}^{(1)}(\alpha)}{V_{g,n}} \prod_{j=1}^n \frac{x_j^{2
        \alpha_j}}{2^{2 \alpha_j} (2 \alpha_j+1)!} \cdot
  \end{equation*}
  We replace $\hat{c}_{g,n}^{(1)}$ by its expression from \cref{prop:cgn_order_two}, and find the claimed
  expression by \cref{lemm:sum_p_k}.  The remainder is
  \begin{equation*}
    \O[n]{\frac{V_{g,n}}{\jap{g}^2} \jap{\x}^3 \exp \div{x_1+\ldots+x_n}}
  \end{equation*}
  because for all $y \geq 0$,
  \begin{align*}
    \sum_{k=1}^{+ \infty} \frac{k^4 y^{2k}}{2^{2k}(2k+1)!}
    \leq y^2 + \sum_{k=2}^{+ \infty} \frac{p_4(k) \, y^{2k}}{2^{2k}(2k+1)!}
    = \O{ \jap{y}^3 \exp \div y}.
  \end{align*}
\end{proof}

\section{Proof of \cref{thm:derivative_coeff}}
\label{sec:proof:coeff}

The aim of this section is to prove \cref{thm:derivative_coeff}, i.e. that for any order $N$, 
\begin{equation*}
  \delta^{\mathbf{m}}c_{g,n}(\alpha) = \O[n,N]{ \jap{\alpha}^{N} \frac{V_{g,n}}{\jap{g}^N}},
\end{equation*}
for any $\mathbf{m} \in \N_0^n$ such that $\abso{\mathbf{m}} \in \{2N-1,2N\}$ and any large enough multi-index $\alpha$.  The proof
relies on Mirzakhani's recursion formula, \cref{theo:rec_mirz}.  In order to be able to apply the discrete differential
operator $\delta$ on its terms \eqref{e:rec_non_sep} and \eqref{e:rec_sep}, we will use the following lemma.

\begin{lem} 
  \label{lem:delta_sum_k1_k2}
  Let $(c_{k_1,k_2})_{k_1,k_2 \geq 0}$ be a family of real numbers, and $v : \N_0 \rightarrow \R$ be the function
  defined by
  \begin{equation*}
    \forall k \geq 0, \quad 
    v_k := \sum_{\substack{k_1 + k_2 = k \\ k_1, k_2 \geq 0}} c_{k_1,k_2}.
  \end{equation*}
  Then, for any integers
  $m \geq 1$ and $k \geq 0$,
  \begin{equation*}
    \delta^m v_k = \sum_{\substack{k_1+k_2=k \\ k_1 \geq k_2}} \delta_1^mc_{k_1,k_2}
    + \sum_{\substack{k_1+k_2=k \\ k_1 < k_2}} \delta_2^mc_{k_1,k_2}
    - \sum_{m_1+m_2=m-1} \delta_1^{m_1} \delta_2^{m_2} c_{\floor{\frac{k+1}{2}},\floor{\frac k2}+1}.
  \end{equation*}
\end{lem}

\begin{proof}
  We prove the formula by induction on the integer $m$. The initialisation at $m=0$ is trivial.  For $m \geq 0$, let us
  assume the property at the rank $m$.  Let~$k \geq 0$ be an integer; we assume that $k=2p+1$ is an odd number (the
  proof when $k$ is even is the same). By definition of the operator $\delta$ and thanks to the induction hypothesis,
  \begin{align*}
    \delta^{m+1} v_k
    = & \delta^m v_k - \delta^m v_{k+1} \\
    = & \sum_{\substack{k_1+k_2=2p+1 \\ k_1 \geq k_2}} \delta_1^mc_{k_1,k_2}
    - \sum_{\substack{k_1+k_2=2p+2 \\ k_1 \geq k_2}} \delta_1^mc_{k_1,k_2} \\
      & + \sum_{\substack{k_1+k_2=2p+1 \\ k_1 < k_2}} \delta_2^mc_{k_1,k_2}
    - \sum_{\substack{k_1+k_2=2p+2 \\ k_1 < k_2}} \delta_2^mc_{k_1,k_2} \\
    & - \sum_{m_1+m_2=m-1} (\delta_1^{m_1} \delta_2^{m_2} c_{p+1,p+1}
      - \delta_1^{m_1} \delta_2^{m_2} c_{p+1,p+2}) \\
    =: & \,S_1 - S_2 + S_3 - S_4 - S_5.
  \end{align*}
  Let us perform a change of indices $k_1'=k_1-1$ in the sum $S_2$, singling out the term of $S_2$ for which $k_1=k_2=p+1$,
  so that we sum over the same set of indices as $S_1$. We obtain:
  \begin{equation*}
    S_1-S_2
    = \Bigg( \sum_{\substack{k_1+k_2=2p+1 \\ k_1 \geq k_2}} \delta_1^{m+1}c_{k_1,k_2} \Bigg) - \delta_1^mc_{p+1,p+1}.
  \end{equation*}
  There is no boundary term when we do the same to $S_3$ and $S_4$, now changing the index~$k_2$:
  \begin{equation*}
    S_3-S_4 = \sum_{\substack{k_1+k_2=2p+1 \\ k_1 < k_2}} \delta_2^{m+1}c_{k_1,k_2}.
  \end{equation*}
  We then observe that $\delta_1^mc_{p+1,p+1} + S_5$ is equal to
  \begin{equation*}
    \delta_1^mc_{p+1,p+1} + \sum_{m_1+m_2=m-1} \delta_1^{m_1} \delta_2^{m_2+1} c_{p+1,p+1}
    = \sum_{m_1+m_2=m} \delta_1^{m_1} \delta_2^{m_2} c_{p+1,p+1}
  \end{equation*}
  which leads  to the claimed expression for $\delta^{m+1}v_k$.
\end{proof}

We can now proceed to the proof of \cref{thm:derivative_coeff}, which we restate here for convenience.

\begin{thm}
  \label{thm:derivative_coeff_ibs}
  There exists an increasing sequence of integers $(a_N)_{N \geq 0}$ satisfying the following.
  For any integers $g \geq 0$, $n \geq 1$ such that $2g-2+n>0$, any multi-indices~$\mathbf{m}, \alpha \in \N_0^n$ such that:
  \begin{itemize}
  \item $\abso{\mathbf{m}}  = m_1 + \ldots + m_n \in \{2N-1,2N\}$
  \item $\forall i, (m_i \neq 0 \Rightarrow \alpha_i \geq a_N)$,
  \end{itemize}
  we have:
  \begin{equation*}
    \abso{\delta^{\mathbf{m}}c_{g,n}(\alpha)} \leq C_{n,N} \, \jap{\alpha}^{N} \frac{V_{g,n}}{\jap{g}^N} 
  \end{equation*}
  where $C_{n,N}>0$ is a constant that depends only on $n$ and $N$.
\end{thm}

\begin{proof}
  The proof is an induction on the integer $N$. The case $N=0$ is trivial: indeed, by Lemma \ref{cor:c_d_decrease},
  \begin{equation*}
    \forall \alpha \in \N_0^n, \quad \abso{c_\alpha^{(g,n)}} \leq V_{g,n}. 
  \end{equation*}

  In order to be able to use Mirzakhani's recursion formula, we observe that the result is trivial when $2g-2+n=1$, for
  any $N>0$. Indeed,
  \begin{itemize}
  \item if $(g,n) = (0,3)$, then $\delta^{\mathbf{m}} c_{0,3}(\alpha) = 0$ for any $\mathbf{m}, \alpha \in \N_0^3$ such that
    $\alpha \neq \z{3}$
  \item if $(g,n) = (1,1)$, then $\delta^m c_{1,1}(\alpha) = 0$ for any $m \geq 0$ and any $\alpha \geq 2$.
  \end{itemize}
  As a consequence, provided that $a_{N} \geq 2$ for $N \geq 1$, the result is automatic.
  
  For an integer $N \geq 0$, let us assume the result to hold at the rank $N$, and prove it at the rank $N+1$.  Let us
  consider integers $g$, $n$ such that $2g-2+n>1$.
  Let $\mathbf{m}, \alpha \in \N_0^n$ be multi-indices, such that:
  \begin{itemize}
  \item $\abso{\mathbf{m}} = m_1 + \ldots + m_n = 2N+1$
  \item $\forall i, (m_i \neq 0 \Rightarrow \alpha_i \geq a_{N+1})$,
  \end{itemize}
  where $a_{N+1}$ is an integer that will be determined during the proof. By symmetry of the volume coefficients, we can
  assume that $m_1>0$.

  Let us write the coefficient $\delta^{\mathbf{m}}c_\alpha^{(g,n)}$ using Mirzakhani's topological recursion formula,
  \cref{theo:rec_mirz}. We obtain:
  \begin{align*}
    \abso{\delta^{\mathbf{m}} c_{g,n}(\alpha)}
    & \leq \sum_{j=2}^n \abso{\delta^{\mathbf{m}} \mathcal{A}^{(j)}_{g,n}(\alpha)} + \abso{\delta^{\mathbf{m}} \mathcal{B}_{g,n}(\alpha)}
      + \sum_{\iota \in \mathcal{I}_{g,n}} \abso{\delta^{\mathbf{m}} \mathcal{C}^{(\iota)}_{g,n}(\alpha)}\\
      & =: (A) + (B) + (C).
  \end{align*}
  We shall estimate these different contributions successively, keeping in mind that the aim is to establish a decay for
  each term at the rate $\jap{\alpha}^{N+1} \frac{V_{g,n}}{\jap{g}^{N+1}}$.

  \begin{proofstep}[Estimate of the term (A)]
    The term $(A)$ is equal to zero if $n=1$, and then there is nothing to be proved. Otherwise, let
    $j \in \{2, \ldots, n\}$. By \cref{e:rec_other_bound},
    \begin{equation*}
      \mathcal{A}^{(j)}_{g,n}(\alpha)
      = 8 \; (2\alpha_j+1) \sum_{i =0}^{+ \infty}  
      u_{i} \; c_{g,n-1}(\tilde \alpha^{(i)})
    \end{equation*}
    where $\tilde{\alpha}^{(i)} := (i+\alpha_1+\alpha_j-1,\alpha_2,\ldots,\hat{\alpha}_j,\ldots,\alpha_n)$. By a change
    of variable in the sum, if we set $u_{-1}=0$, then
    \begin{equation*}
      \delta_1 \mathcal{A}^{(j)}_{g,n}(\alpha)
      = 8 \; (2\alpha_j+1) \sum_{i =0}^{+ \infty}  
      (u_{i}-u_{i-1}) \; c_{g,n-1}(\tilde \alpha^{(i)}).
    \end{equation*}
    \begin{itemize}
    \item Let us first treat the case when $m_j = 0$. By applying the discrete derivatives $\delta_1^{m_1-1}$ and
      $\delta_i^{m_i}$ for $i \notin \{1, j\}$, we observe that
      \begin{equation*}
        \delta^{\mathbf{m}} \mathcal{A}^{(j)}_{g,n}(\alpha)
        =  8 \; (2\alpha_j+1) \sum_{i =0}^{+ \infty}  
               (u_{i}-u_{i-1}) \; \delta^{\tilde{\mathbf{m}}}
               c_{g,n-1}(\tilde \alpha^{(i)})
      \end{equation*}
      for $\tilde{\mathbf{m}} = (m_1-1, m_2, \ldots, \hat{m}_j, \ldots, m_n)$.  Then, the bound on $u_i-u_{i-1}$
      from \cref{lemm:behaviour_u_i} implies the existence of a universal constant $C>0$ such that
      \begin{equation*}
        \label{eq:delta_m_Aj}
        \abso{\delta^{\mathbf{m}} \mathcal{A}^{(j)}_{g,n}(\alpha)}
        \leq C \jap{\alpha} \sum_{i =0}^{+ \infty}  
        4^{-i} \; \abso{\delta^{\tilde{\mathbf{m}}} c_{g,n-1}(\tilde \alpha^{(i)})}.
      \end{equation*}

      We now want to use the induction hypothesis to bound $\delta^{\tilde{\mathbf{m}}} c_{g,n-1}(\tilde \alpha^{(i)})$,
      for every $i \geq 0$. We observe that $\abso{\tilde{\mathbf{m}}} = \abso{\mathbf{m}}-1 = 2N$, and decide to choose
      the parameter $a_{N+1}$ so that $a_{N+1}>a_N$. Then,
      \begin{equation*}
      i+\alpha_1+\alpha_j-1 \geq \alpha_1 - 1 \geq a_N,
    \end{equation*}
    and the multi-indices $\tilde{\mathbf{m}}$, $\tilde{\alpha}^{(i)}$ therefore satisfy the hypotheses of the theorem
    at the rank $N$. Hence,
      \begin{equation*}
        \abso{\delta^{\tilde{\mathbf{m}}} c_{g,n-1}(\tilde \alpha^{(i)})}
        \leq C_{n-1,N} \, \jap{\tilde{\alpha}^{(i)}}^N\frac{V_{g,n-1}}{\jap{g}^N} \cdot  
      \end{equation*}
      By \cref{e:volume_n_plus_one}, this implies that
      \begin{equation*}
        \delta^{\tilde{\mathbf{m}}} c_{g,n-1}(\tilde \alpha^{(i)})
        = \O[n,N]{\jap{\alpha}^N \jap{i}^N \frac{V_{g,n}}{\jap{g}^{N+1}}},
      \end{equation*}
      from which we deduce that, as soon as $m_j = 0$,
      \begin{align*}
        \abso{\delta^{\mathbf{m}} \mathcal{A}^{(j)}_{g,n}(\alpha)}
        & = \O[n,N]{\jap{\alpha}^{N+1} \frac{V_{g,n}}{\jap{g}^{N+1}} \sum_{i=0}^{+ \infty} 4^{-i}  \jap{i}^N} \\
        & = \O[n,N]{\jap{\alpha}^{N+1} \frac{V_{g,n}}{\jap{g}^{N+1}}},
      \end{align*}
      which is precisely our claim.
      \item Now, if $m_j>0$, we need to be more careful when applying the derivative~$\delta_j$ because of the
        dependance in $\alpha_j$ of $\mathcal{A}_{g,n}^{(j)}(\alpha)$. We prove by a simple induction that
        \begin{align*}
          \delta^{\mathbf{m}} \mathcal{A}^{(j)}_{g,n}(\alpha)
          = \, & 8 \; (2\alpha_j+1) \sum_{i =0}^{+ \infty}  
                 (u_{i}-u_{i-1}) \; \delta^{\tilde{\mathbf{m}}}
                 c_{g,n-1}(\tilde \alpha^{(i)}) \\
               & - 16 \, m_j \sum_{i =0}^{+ \infty}
                 (u_{i}-u_{i-1}) \; \delta^{\hat{\mathbf{m}}}
                 c_{g,n-1}(\tilde \alpha^{(i+1)})
        \end{align*}
        where $\tilde{\mathbf{m}}$ is as before and
        $\hat{\mathbf{m}} := (m_1+m_j-2, m_2, \ldots, \hat{m}_j, \ldots, m_n)$. We observe that
        $\abso{\hat{\mathbf{m}}} = 2N-1$, and this allows us to apply the induction hypothesis to this additional
        term. The same computation as in the case $m_j=0$ leads to the same bound, since $m_j = \O[N]{1}$.
      \end{itemize}
      We sum up the $n-2 = \O[n]{1}$ contributions for $j \in \{2, \ldots, n\}$ and conclude that
      \begin{equation}
        \label{e:A}
        \sum_{j=2}^n \abso{\delta^{\mathbf{m}} \mathcal{A}^{(j)}_{g,n}(\alpha)}
        = \O[n,N]{\jap{\alpha}^{N+1}\frac{V_{g,n}}{\jap{g}^{N+1}}}.
      \end{equation}
    \end{proofstep}
    
    \begin{proofstep}[Estimate of the term (B)]
      Let us first observe that this term only appears whenever $g \geq 1$.
      As in the case (A), we start by writing that by \cref{e:rec_non_sep},
      \begin{equation*}
        \delta_1 \mathcal{B}_{g,n}(\alpha)
        =  16  \sum_{i=0}^{+ \infty} \sum_{k_1+k_2=i+\alpha_1-2}
        (u_i-u_{i-1}) \; c_{g-1,n+1}(\tilde{\alpha}^{(k_1,k_2)})
      \end{equation*}
      where $\tilde{\alpha}^{(k_1,k_2)} = (k_1,k_2,\alpha_2,\ldots,\alpha_n)$.
      However, this time, the dependency on~$\alpha_1$ is more complex, and we need to use \cref{lem:delta_sum_k1_k2}
      to apply the operator $\delta_1^{m_1-1}$ to the equation. We obtain:
      \begin{align}
        \abso{\delta^{\mathbf{m}} \mathcal{B}_{\alpha}^{(g,n)}}
        \leq & \, C  \sum_{i=0}^{+ \infty} \sum_{\substack{k_1+k_2=i+\alpha_1-2 \\ k_1 \geq k_2}}
        4^{-i} \; \abso{\delta_1^{m_1-1} \delta^{\tilde{\mathbf{m}}} c_{g-1,n+1}(\tilde{\alpha}^{(k_1,k_2)})} \label{eq:delta_m_B1} \\
             & + C \sum_{i=0}^{+ \infty} \sum_{\substack{k_1+k_2=i+\alpha_1-2 \\ k_1 < k_2}}
        4^{-i} \; \abso{\delta_2^{m_1-1} \delta^{\tilde{\mathbf{m}}} c_{g-1,n+1}(\tilde{\alpha}^{(k_1,k_2)})}  \label{eq:delta_m_B2} \\ 
             & + C \sum_{\mu_1+ \mu_2= m_1-2} \sum_{i=0}^{+\infty}  4^{-i}
               \abso{\delta_{1}^{\mu_1}\delta_2^{\mu_2} \delta^{\tilde{\mathbf{m}}}
               c_{g-1,n+1}(\tilde{\alpha}^{(\floor{\frac{i+\alpha_1-1}{2}},\floor{\frac{i+\alpha_1}{2}})})},
               \label{eq:delta_m_B3}
      \end{align}
      where $\tilde{\mathbf{m}} = (0,0,m_2, \ldots, m_n) \in \N_0^{n+1}$, and the universal constant $C>0$ comes once
      again from \cref{lemm:behaviour_u_i}.  We estimate each term successively, using the induction hypothesis.
      
      \begin{itemize}
      \item Let us assume that the parameter $a_{N+1}$ is $\geq 2a_N+2$. Then, by hypothesis, $\alpha_1 \geq 2 a_N+2$,
        and therefore for any $i \geq 0$ and any $k_1$, $k_2$ in the $i$-th term of the sum \eqref{eq:delta_m_B1},
        \begin{equation*}
          k_1 \geq \frac{k_1+k_2}{2} = \frac{i + \alpha_1 - 2}{2} \geq a_N.
        \end{equation*}
        We can then apply the induction hypothesis to $(m_1-1, \z{n}) + \tilde{\mathbf{m}}$, of $\ell^1$-norm~$2N$, and
        $\tilde{\alpha}^{(k_1,k_2)}$. This yields
        \begin{align*}
          \abso{\delta_1^{m_1-1}\delta^{\tilde{\mathbf{m}}} c_{g-1,n+1}(\tilde{\alpha}^{(k_1,k_2)})}
          & \leq C_{n+1,N} \, \jap{\tilde \alpha^{(k_1,k_2)}}^N\frac{V_{g-1,n+1}}{\jap{g-1}^N}  \\
          & = \O[n,N]{\jap{\tilde \alpha^{(k_1,k_2)}}^N\frac{V_{g,n}}{\jap{g}^{N+1}}}
        \end{align*}
        since $g \geq 1$, and by \cref{e:volume_n_plus_one,e:volume_same_euler}. We then use the fact that
        \begin{equation}
          \label{eq:sum_u_i_k}
          \sum_{i = 0}^{+ \infty} \sum_{\substack{k_1+k_2=i+\alpha_1-2 \\ k_1 \geq k_2}} 4^{-i} \jap{\tilde \alpha^{(k_1,k_2)}}^N = \O[N]{\jap{\alpha}^{N+1}},
        \end{equation}
        to conclude that if $a_{N+1} \geq 2 a_N+2$, then the term~\eqref{eq:delta_m_B1} is
        \begin{equation*}
          \O[n,N]{\jap{\alpha}^{N+1}\frac{V_{g,n}}{\jap{g}^{N+1}}}.
        \end{equation*}
      \item By symmetry of the coefficients, the term \eqref{eq:delta_m_B2} is equal to
        \begin{equation*}
          \sum_{i=0}^{+ \infty} \sum_{\substack{k_1+k_2 =i+\alpha_1-2 \\ k_2 > k_1}}
          4^{-i} \; \abso{\delta_1^{m_1-1} \delta^{\tilde{\mathbf{m}}} c_{g-1,n+1}(\tilde \alpha^{(k_2,k_1)})}
        \end{equation*}
        is therefore smaller than the term \eqref{eq:delta_m_B1}.
      \item For the term \eqref{eq:delta_m_B3}, we observe that since $\alpha_1 \geq 2 a_N + 2$, for all $i \geq 0$,
        \begin{equation*}
          \floor*{\frac{i + \alpha_1}{2}}
          \geq \floor*{\frac{i + \alpha_1 - 1}{2}}
          \geq \frac{\alpha_1 - 2}{2} \geq a_N.
        \end{equation*}
        Furthermore, for any integers such that $\mu_1+\mu_2=m_1-2$, the norm of
        $(\mu_1, \mu_2, \z{n-1}) + \tilde{\mathbf{m}}$ is equal to $2N-1$, and therefore, by the induction hypothesis,
        \begin{multline*}
          \abso{\delta_{1}^{\mu_1}\delta_2^{\mu_2} \delta^{\tilde{\mathbf{m}}}
            c_{g-1,n+1}(\tilde{\alpha}^{(\floor{\frac{i+\alpha_1-1}{2}},\floor{\frac{i+\alpha_1}{2}})})} \\
          \leq C_{n+1,N} \, \jap{\tilde{\alpha}^{(\floor{\frac{i+\alpha_1-1}{2}},\floor{\frac{i+\alpha_1}{2}})}}^N\frac{V_{g-1,n+1}}{\jap{g-1}^N} ,
    \end{multline*}
    and \eqref{eq:delta_m_B3} hence satisfies the same bound as the other terms.
  \end{itemize}
  As a conclusion, provided that $a_{N+1} \geq 2 a_N+2$, 
  \begin{equation*}
    (B) = \abso{\delta^{\mathbf{m}} \mathcal{B}_{g,n}(\alpha)} = \O[n,N]{\jap{\alpha}^{N+1} \frac{V_{g,n}}{\jap{g}^{N+1}}}.
  \end{equation*}
\end{proofstep}

\begin{proofstep}[Estimate of the term (C)]
  For the term (C), similarly, by \cref{e:rec_sep}, for every configuration $\iota = (g_1,I,J)$ where
  $g_1+g_2 = g$ and $I \sqcup J = \{2, \ldots, n\}$, if we denote $n_1=|I|$ and $n_2=|J|$,
  \begin{equation*}
   \delta_1 \mathcal{C}^{(\iota)}_{g,n}(\alpha)
    =  16 \sum_{i=0}^{+ \infty} \sum_{k_1+k_2=i+\alpha_1-2}
    (u_{i}-u_{i-1}) \;  c_{g_1,n_1+1}(\tilde \alpha_I^{(k_1)})
    \; c_{g_2,n_2+1}(\tilde \alpha_J^{(k_2)})
  \end{equation*}
  where $\tilde \alpha_I^{(k_1)} = (k_1, \alpha_I)$ and $\tilde \alpha_J^{(k_2)} = (k_2, \alpha_J)$. As before, we prove
  that
  \begin{align}
    & \abso{\delta^{\mathbf{m}} \mathcal{C}^{(\iota)}_{g,n}(\alpha)} \\
    &  \leq  C \sum_{i=0}^{+ \infty} \sum_{\substack{k_1+k_2= \\ i+\alpha_1-2 \\ k_1 \geq k_2}}
    4^{-i}  \abso{\delta^{(m_1-1,\mathbf{m}_I)} c_{g_1,n_1+1}(\tilde \alpha_I^{(k_1)})}
    \abso{\delta^{(0,\mathbf{m}_J)} c_{g_2,n_2+1}(\tilde \alpha_J^{(k_2)})} \label{eq:upper_bound_C_1}\\
    & + C \sum_{i=0}^{+ \infty} \sum_{\substack{k_1+k_2  = \\ i+\alpha_1-2 \\ k_1 < k_2}}
    4^{-i}  \abso{\delta^{(0,\mathbf{m}_I)} c_{g_1,n_1+1}(\tilde \alpha_I^{(k_1)})}
    \abso{\delta^{(m_1-1,\mathbf{m}_J)} c_{g_2,n_2+1}(\tilde \alpha_J^{(k_2)})} \\
    & + C \sum_{\substack{\mu_1+\mu_2 \\ = m_1-2 \\ i \geq 0}}
    \abso{\delta^{(\mu_1,\mathbf{m}_I)}c_{g_1,n_1+1}(\tilde \alpha_I^{(\floor{\frac{i+\alpha_1-1}{2}})})}
    \abso{\delta^{(\mu_2,\mathbf{m}_J)}c_{g_2,n_2+1}(\tilde \alpha_I^{(\floor{\frac{i+\alpha_1}{2}})})}. \label{eq:upper_bound_C_3}
  \end{align}
  
  We now estimate the term \eqref{eq:upper_bound_C_1} using the induction hypothesis on the two terms
  $\delta^{(m_1-1,\mathbf{m}_I)} c_{g_1,n_1+1}(\tilde \alpha_I^{(k_1)})$ and
  $\delta^{(0,\mathbf{m}_J)} c_{g_2,n_2+1}(\tilde \alpha_J^{(k_2)})$. Let us set
  \begin{equation*}
    N_1 := \floor*{\frac{m_1+\abso{\mathbf{m}_I}}{2}} \leq N
    \quad \text{and} \quad
    N_2 := \floor*{\frac{\abso{\mathbf{m}_J}+1}{2}} \leq N
  \end{equation*}
  so that $m_1-1 + \abso{\mathbf{m}_I} \in \{2N_1-1,2N_1\}$ and $\abso{\mathbf{m}_J} \in \{2N_2-1,2N_2\}$.
  Then, we observe that under the hypothesis $a_{N+1} \geq 2a_N+2$, for any term in \cref{eq:upper_bound_C_1},
  $k_1 \geq a_N \geq a_{N_1}$. We can therefore apply the induction hypothesis at the rank $N_1$ and obtain
  \begin{equation*}
    \delta^{(m_1-1,\mathbf{m}_I)} c_{g_1,n_1+1}(\tilde \alpha_I^{(k_1)})
    = \O[n_1,N_1]{\jap{\tilde \alpha_I^{(k_1)}}^{N_1} \frac{V_{g_1,n_1+1}}{\jap{g_1}^{N_1}}}.
  \end{equation*}
  We also have that
  \begin{equation*}
    \delta^{(0,\mathbf{m}_J)} c_{g_2,n_2+1}(\tilde \alpha_J^{(k_2)})
    = \O[n_2,N_2]{\jap{\tilde \alpha_J^{(k_2)}}^{N_2}\frac{V_{g_2,n_2+1}}{\jap{g_2}^{N_2}}}
  \end{equation*}
  (note that there is no condition on the index $k_2$ because there is no derivative w.r.t. the first variable in
  $\delta^{(0,\mathbf{m}_J)}$). We obtain by the same method as before that the term~\eqref{eq:upper_bound_C_1} is
  \begin{equation}
    \label{eq:bound_C_1_g1_big}
    \mathcal{O}_{n,N_1,N_2} 
      \paren*{\jap{\alpha}^{N_1+N_2+1} \frac{V_{g_1,n_1+1} V_{g_2,n_2+1}}{\jap{g_1}^{N_1} \jap{g_2}^{N_2}}}.    
  \end{equation}
  
  We then wish to apply \cref{lem:volume_sum_K} in order to bound the sum over all configurations. This lemma implies
  that
  \begin{equation}
    \label{e:appli_lemma_n12}
    \sum_{\substack{\iota \in \mathcal{I}_{g,n} \\ 2 g_i + n_i > N_i+1}} \frac{V_{g_1,n_1+1} V_{g_2,n_2+1}}{\jap{g_1}^{N_1} \jap{g_2}^{N_2}}
    = \O[n,N]{\frac{V_{g,n-1}}{\jap{g}^{N_1+N_2+1}}}
    = \O[n,N]{\frac{V_{g,n}}{\jap{g}^{N+2}}},
  \end{equation}
  by \cref{e:volume_n_plus_one} since $n_1+n_2 = n-1$ for any $\iota \in \mathcal{I}_{g,n}$, and because
  \begin{equation*}
    N_1+N_2
    = \floor*{\frac{m_1+\abso{\mathbf{m}_I}}{2}} + \floor*{\frac{\abso{\mathbf{m}_J}+1}{2}}
    \in \{N,N+1\}.
  \end{equation*}
  
  As a consequence, in order to conclude, we need to be able to restrict the sum over $\iota \in \mathcal{I}_{g,n}$ to
  the configurations such that $2g_i+n_i>N_i+1$ for $i \in \{1, 2\}$. This is achieved by adding a new constraint on the
  parameter~$a_{N+1}$: we assume that $a_{N+1} \geq 3(2N+1)$.  Thanks to this additional hypothesis, we can prove that,
  for all configuration $\iota \in \mathcal{I}_{g,n}$,
  \begin{itemize}
  \item either $2g_i+n_i> N_i+1$ for $i = 1$ and $2$;
  \item or $2g_1 + n_1 \leq N_1+1$, in which case
    $\delta^{(m_1-1, \mathbf{m}_I)} c_{g_1,n_1+1}(\tilde{\alpha}_I^{(k_1)}) = 0$ for any integers $k_1 \geq k_2$ such
    that $k_1+k_2 \geq \alpha_1-2$;
  \item or $2g_2 + n_2 \leq N_2+1$, in which case
    $\delta^{(0, \mathbf{m}_J)} c_{g_2,n_2+1}(\tilde{\alpha}_J^{(k_2)}) = 0$ for any integer $k_2 \geq 0$.
  \end{itemize}
  Provided this claim is proved, we can then say that the sum over all configurations~$\iota$ of the term
  \eqref{eq:upper_bound_C_1} is equal to the sum over all $\iota$ such that $2g_1+n_i>N_i+1$, which then is
  \begin{equation*}
    \O[n,N]{\jap{\alpha}^{N+2} \frac{V_{g,n}}{\jap{g}^{N+2}}}
  \end{equation*}
  by \cref{eq:bound_C_1_g1_big,e:appli_lemma_n12}.  This implies that the sum \eqref{eq:upper_bound_C_1} satisfies the
  claimed estimate for any $\alpha$ such that $\abso{\alpha} \leq 3g-3+n$. Otherwise, because of the degree of
  $V_{g,n}(\x)$, the sum \eqref{eq:upper_bound_C_1} is equal to zero and the estimate trivially holds.

  Let us now prove our claim.
  \begin{itemize}
  \item First, if $2g_1+n_1 \leq N_1+1$, then for any $k_1 \geq k_2$ such that $k_1+k_2\geq\alpha_1-2$, on the one hand, 
    \begin{align*}
      k_1+\abso{\alpha_I}
      & \geq \frac{k_1+k_2}{2} + \abso*{\alpha_I}
      \geq \frac{\alpha_1 + \abso*{\alpha_I}}{2} -1 \\
      & \geq \frac{a_{N+1}}{2} \#\{ i \in \{ 1 \} \cup I \, : \, m_i \neq 0 \} -1
    \end{align*}
    by hypothesis on $\alpha$. On the other hand,
    \begin{equation*}
      \#\{ i \in \{ 1 \} \cup I \, : \, m_i \neq 0 \}
      \geq \frac{m_1 + \abso{\mathbf{m}_I}}{\abso{\mathbf{m}}_\infty}
      \geq \frac{2 N_1}{2N+1} \cdot
    \end{equation*}
    We use the hypothesis $a_{N+1} \geq 3(2N+1)$ to deduce that
    \begin{equation*}
      k_1 + \abso{\alpha_I} \geq 3 N_1 -1 \geq 6g_1+3n_1-4 > 3g_1-3 + (n_1+1),
    \end{equation*}
    because $3g_1+2n_1 > 2$. The latter quantity is the degree of the polynomial $V_{g_1,n_1+1}(\x)$ in the variables
    $x_1^2, \ldots, x_{n_1+1}^2$, and therefore the previous inequality implies that $\delta^{(m_1-1,\mathbf{m}_I)}c_{g_1,n_1+1}(k_1,\alpha_I) = 0$.
  \item Similarly, we prove that for any $k_2 \geq 0$,
    \begin{equation*}
      k_2 + \abso{\alpha_J} \geq  \frac{a_{N+1} \abso{\mathbf{m}_J}}{\abso{\mathbf{m}}_\infty} \geq 3(2N_2-1)
    \end{equation*}
    and therefore if $2g_2+n_2 \leq N_2+1$, then
    \begin{equation*}
      k_2 + \abso{\alpha_J}
      \geq 12g_2+6n_2-9 > 3g_2 -3 + (n_2 +1)
    \end{equation*}
    and hence $\delta^{(0,\mathbf{m}_J)}c_{g_2,n_2+1}(k_2,\alpha_J) = 0$.
  \end{itemize}

  The estimate of the term \eqref{eq:upper_bound_C_3} is the same: we apply the induction hypothesis to
  $\delta^{(\mu_1,\mathbf{m}_I)}c_{g_1,n_1+1}(\tilde \alpha_I^{(\floor{\frac{i+\alpha_1-1}{2}})})$ and
  $\delta^{(\mu_2,\mathbf{m}_J)}c_{g_2,n_2+1}(\tilde \alpha_I^{(\floor{\frac{i+\alpha_1}{2}})})$, at the admissible ranks
  \begin{equation*}
    N_1 := \floor*{\frac{\mu_1 + \abso{\mathbf{m}_I}+1}{2}}
    \quad \text{and} \quad
    N_2 := \floor*{\frac{\mu_2 + \abso{\mathbf{m}_J}+1}{2}}.
  \end{equation*}
  We observe that $N_1+N_2= N$, and this therefore yields the claimed result.
\end{proofstep}
As a conclusion, we have proved that under the hypotheses $a_{N+1} \geq 2 a_N+2$ and $a_{N+1} \geq 3(2N+1)$, for any
multi-index $\mathbf{m}$ of norm $\abso{\mathbf{m}}=2N+1$ and any multi-index $\alpha$ such that
$\forall i, (m_i \neq 0 \Rightarrow \alpha_i \geq a_{N+1})$,
\begin{equation*}
  \abso{\delta^{\mathbf{m}} c_{g,n}(\alpha)} \leq (A) + (B) + (C)
  \leq C_{n,N+1} \, \jap{\alpha}^{N+1} \frac{V_{g,n}}{\jap{g}^{N+1}}\cdot
\end{equation*}
This implies the result for any multi-index $\mathbf{m}$ of norm $2N+2$ too, simply because for any sequence
$(v(\alpha))_\alpha$, if $m_1>0$ for instance, then for all $\alpha$,
\begin{equation*}
  \abso{\delta^{\mathbf{m}} v(\alpha)} \leq \abso{\delta^{(m_1-1, m_2, \ldots, m_n)}v(\alpha)}
  + \abso{\delta^{(m_1-1, m_2, \ldots, m_n)}v(\alpha_1+1, \alpha_2, \ldots, \alpha_n)}.
\end{equation*}
This concludes the induction. 
\end{proof}

\section{Proof of \cref{theo:volume_asympt_exp}}
\label{sec:funct-ultim-polyn}

\subsection{Discrete Taylor expansion}

\cref{thm:derivative_coeff} states that the function $\alpha \mapsto c_{g,n}(\alpha)$ has small derivatives for large
enough values of $\alpha$. Had we proved that the derivatives are small for \emph{any}~$\alpha$, we could have used a
discrete version of the Taylor--Lagrange formula, such as the one below, to conclude that
$\alpha \mapsto c_{g,n}(\alpha)$ is well-approximated by polynomial functions.

\begin{lem}[Discrete Taylor--Lagrange formula]
  \label{lemm:taylor_zero}
  Let $n \geq 1$ and $f : \N_0^n \rightarrow \R$. We assume that there exist a real number $M \geq 0$ and integers $K, p \geq 0$ such
  that, for any multi-index $\mathbf{m}$ of norm $\abso{\mathbf{m}} = K+1$,
  \begin{equation*}
    \forall \alpha \in \N_0^n, \quad \abso{\delta^{\mathbf{m}} f (\alpha)} \leq M \jap{\alpha}^p.
  \end{equation*}
  Then, there exists a polynomial function $\tilde{f}^{(K)} : \N_0^n \rightarrow \R$ of degree at most $K$ such that
  \begin{equation*}
    \forall \alpha \in \N_0^n, \quad
    \abso{f(\alpha) - \tilde{f}^{(K)} (\alpha)} \leq M n^{K+1} \jap{\alpha}^{p+K+1}.
  \end{equation*}
  Furthermore, the coefficients of the polynomial function $\tilde{f}^{(K)}$ can be expressed as linear combinations of the
  derivatives $\delta^{\mathbf{m}}f(\z{n})$, for multi-indices $\mathbf{m} \in \N_0^n$ of norm $\abso{\mathbf{m}} \leq K$.
\end{lem}

\begin{proof}
  We proceed by induction on the integer $K$.

  For $K=0$, we observe that by \cref{lemm:discrete_integral}, for all $\alpha$,
  \begin{equation*}
    \abso{f(\alpha) - f(\z{n})} 
    \leq \sum_{i=1}^n \sum_{k=0}^{\alpha_i-1} \abso{\delta_i f(\z{i-1},k,\alpha_{i+1}, \ldots, \alpha_n)}
    \leq M n \jap{\alpha}^{p+1},
  \end{equation*}
  so the result holds if we take $\tilde{f}^{(0)}$ to be the constant function equal to $f(\z{n})$.
  
  Let us now assume the result at a rank $K-1$ for a $K \geq 1$, and deduce the result at the rank $K$.  For any integer
  $i \in \{1, \ldots, n\}$, the function $\delta_i f$ satisfies the induction hypothesis at the rank $K-1$. Hence, there
  exists a polynomial function $\tilde{f}_i^{(K-1)}$ of degree at most $K-1$, and whose coefficients can be expressed as
  linear combinations of the $\delta^{\mathbf{m}} \delta_i f(\z{n})$ for $\abso{\mathbf{m}} \leq K-1$, such that
  \begin{equation*}
    \forall \alpha \in \N_0^n, \quad
    \abso{\delta_i f(\alpha) - \tilde{f}_i^{(K-1)}(\alpha)}
    \leq M n^K \jap{\alpha}^{p+K}.
  \end{equation*}
  Inspired by the discrete integration formula (\cref{lemm:discrete_integral}), we define
  \begin{equation*}
    \tilde{f}^{(K)}(\alpha) := f(\z{n}) - \sum_{i=1}^n \sum_{k=0}^{\alpha_i-1} \tilde{f}_i^{(K-1)}(\z{i-1}, k, \alpha_{i+1}, \ldots, \alpha_n).
  \end{equation*}
  We notice that $\tilde{f}^{(K)}$ is a polynomial of degree at most $K$, and its coefficients are linear combinations
  of $f(\z{n})$ and the coefficients of $(\tilde{f}_i)_i$, and therefore linear combinations of the
  $\delta^{\mathbf{m}}f(\z{n})$ for $\abso{\mathbf{m}} \leq K$. By \cref{lemm:discrete_integral}, for any multi-index
  $\alpha \in \N_0^n$,
  \begin{align*}
    & \abso{f(\alpha) - \tilde{f}^{(K)}(\alpha)} \\
    & \leq \sum_{i=1}^n \sum_{k=0}^{\alpha_i-1} \abso{\delta_if(\z{i-1}, k, \alpha_{i+1}, \ldots, \alpha_n)
      - \tilde{f}_i^{(K-1)}(\z{i-1}, k, \alpha_{i+1}, \ldots, \alpha_n)}\\
    & \leq M n^{K+1} \jap{\alpha}^{p+K+1},
  \end{align*}
  and the conclusion follows.
\end{proof}

However, we can expect from the second-order approximation, \cref{prop:cgn_order_two}, that the function
$\alpha \in \N_0^n \mapsto c_{g,n}(\alpha)$ is \emph{not} well-approximated by polynomial functions, but rather by a
combination of polynomial functions and indicator functions, correcting the values of the function for small
$\alpha$. The aim of the following section is to define such a class of functions, and prove a shifted Taylor--Lagrange
estimate in this new setting.

\subsection{Functions ultimately polynomial in each variable}

\begin{lem}[and Definition]
  For any integers $n \geq 1$ and $K, a \geq 0$, the two following families of functions  $\N_0^n \rightarrow \R$, 
  \begin{itemize}
  \item functions of the form
    \begin{equation*}
      \alpha \mapsto \prod_{i \in I} \alpha_{i}^{k_i} \, \prod_{i \notin I} \1{\alpha_i = \beta_i}
    \end{equation*}
    where $I \subseteq \{1, \ldots, n\}$, $\mathbf{k} = (k_i)_{i \in I}$ is a multi-index of norm
    $\abso{\mathbf{k}} \leq K$, and $\beta = (\beta_i)_{i \notin I}$ satisfies $\abso{\beta}_\infty < a$;
  \item functions of the form
    \begin{equation*}
      \alpha \mapsto \prod_{i \in I} \alpha_i^{k_i} \1{\alpha_i \geq a} \, \prod_{i \notin I} \1{\alpha_i = \beta_i}
    \end{equation*}
    where $I, \mathbf{k}$ and $\beta$ are defined the same way as in the first point;
  \end{itemize}
  generate the same linear subspace of the space of functions $\N_0^n \rightarrow \R$.  We denote this space as
  $\mathcal{P}_{n,K,a}$, and call its elements \emph{polynomials (of degree at most $K$) in each variable greater than
    $a$}.
\end{lem}
\begin{proof}
  The equivalence of these two definitions comes from the simple observation that for any integers $a,\alpha \geq 0$,
  \begin{equation*}
    1 = \1{\alpha \geq a} + \sum_{\beta=0}^{a-1} \1{\alpha = \beta}.
  \end{equation*}
\end{proof}

Then, elements of $\mathcal{P}_{n,K,a}$ are exactly the kind of functions we imagine the coefficients
$\alpha \mapsto c_{g,n}(\alpha)$ to be well-approximated by: since the derivatives vanish for large enough $\alpha$,
beyond a few small values, the functions are approximated by polynomials.

\subsection{Shifted discrete Taylor expansion}

Let us prove the following shifted Taylor--Lagrange lemma.

\begin{lem}
  \label{lemm:taylor_shift}
  Let $n \geq 1$ be an integer and $f : \N_0^n \rightarrow \R$. We assume that there exists a real number $M \geq 0$ and
  integers $K, a, p \geq 0$ satisfying the following. For any multi-indices $\mathbf{m}, \alpha \in \N_0^n$ such that:
  \begin{itemize}
  \item $\abso{\mathbf{m}} = K+1$,
  \item $\forall i, (m_i \neq 0 \Rightarrow \alpha_i \geq a)$,
  \end{itemize}
  we have
  \begin{equation*}
    \abso{\delta^{\mathbf{m}} f (\alpha)} \leq M \jap{\alpha}^p.
  \end{equation*}
  Then, there exists a function $\tilde{f}^{(K)} \in \mathcal{P}_{n,K,a}$ such that
  \begin{equation*}
    \forall \alpha \in \N_0^n, \quad
    \abso{f(\alpha) - \tilde{f}^{(K)} (\alpha)} \leq C_{n,a,p,K}M \jap{\alpha}^{p+K+1}
  \end{equation*}
  where $C_{n,a,p,K} = 2^{\frac{p}{2}+n} a^{n} \jap{2na}^p  n^{K+1}$.
  The coefficients of $\tilde{f}^{(K)}$ can be expressed as linear combinations of the values $\delta^{\mathbf{m}}f(\alpha)$
  for multi-indices $\alpha, \mathbf{m} \in \N_0^n$ such that $\abso{\alpha}_\infty \leq a$ and $\abso{\mathbf{m}}
  \leq K$.
\end{lem}

\begin{proof}
  The idea is to decompose $\N_0^n$ into subsets on which all of the variables are greater than $a$. More precisely,
  we notice that
  \begin{equation*}
    1 = \sum_{I \subset \{1, \ldots, n\}} \sum_{\substack{(\beta_i)_{i \notin I} \\ \abso{\beta}_\infty < a}}
    \prod_{i \in I} \1{\alpha_i \geq a} \prod_{i \notin I} \1{\alpha_i = \beta_i}.
  \end{equation*}
  Then, we can rewrite the function $f$ as
  \begin{equation}
    \label{e:f_shift}
    f(\alpha) = \sum_{\substack{I \subset \{1, \ldots, n\} \\ = \{i_1 < \ldots < i_r\}}}
      \sum_{\substack{(\beta_i)_{i \notin I} \\ \abso{\beta}_\infty < a}}
    g_{I,\beta}(\alpha_{i_1} - a, \ldots, \alpha_{i_r} - a) \prod_{i \in I} \1{\alpha_i \geq a} \prod_{i \notin I} \1{\alpha_i = \beta_i}
  \end{equation}
  where $g_{I,\beta} : \N_0^r \rightarrow \R$ is defined by setting, for $\hat{\alpha} \in \N_0^r$,
  $g_{I,\beta}(\hat{\alpha}) := f(\alpha)$ where
  \begin{equation*}
    \forall i, \alpha_i :=
    \begin{cases}
      \hat{\alpha}_{k} + a & \text{if } i = i_k \text{ for a } k \in \{1, \ldots, r\}\\
      \beta_i & \text{if } i \notin I.
    \end{cases}
  \end{equation*}

  We wish to apply \cref{lemm:taylor_zero} to the function $g_{I,\beta}$. In order to do so, we need to prove
  an estimate on the derivatives $\delta^{\hat{\mathbf{m}}} g_{I,\beta}(\hat \alpha)$ for any
  $\hat{\mathbf{m}}, \hat \alpha \in \N_0^r$ such that $\abso{\hat{\mathbf{m}}} = K+1$. This will follow from the hypothesis on the function $f$.

  Indeed, we observe that, to any multi-index $\hat{\mathbf{m}} \in \N_0^r$ of norm $K+1$, we can associate a
  multi-index $\mathbf{m} \in \N_0^n$ also of norm $K+1$ by setting
  \begin{equation*}
    \forall i, m_i :=
    \begin{cases}
      \hat{m}_k & \text{if } i = i_k \text{ for a } k \in \{1, \ldots, r\}\\
      0 & \text{if } i \notin I.
    \end{cases}
  \end{equation*}
  Then, for any multi-indices $\hat{\mathbf{m}}, \hat{\alpha} \in \N_0^r$, the corresponding multi-indices $\mathbf{m}, \alpha$ automatically satisfy:
  \begin{equation*}
  \forall i, (m_i\neq 0 \Rightarrow i \in I \Rightarrow \alpha_i \geq a),
\end{equation*}
  and therefore, by hypothesis on $f$,
  \begin{equation*}
    \abso{\delta^{\hat{\mathbf{m}}} g_{I, \beta}(\hat{\alpha})}
    = \abso{\delta^{\mathbf{m}} f(\alpha)}
    \leq M \jap{\alpha}^p
    \leq M 2^{\frac{p}{2}} \jap{2na}^p \jap{\hat{\alpha}}^p
  \end{equation*}
  because 
  $\abso{\alpha} = \abso{\hat{\alpha}} + ra + \abso{\beta} \leq 2n a + \abso{\hat{\alpha}}$
  and for any $x, y$, $\jap{x+y} \leq \sqrt{2} \jap{x} \jap{y}$.

  We can therefore apply \cref{lemm:taylor_zero} to $g_{I,\beta}$, and deduce the existence of a polynomial
  $\tilde g_{I,\beta}^{(K)}$ in $r$ variables, of degree at most $K$, such that
  \begin{equation}
    \label{e:induc_f_beta_I}
    \forall \hat{\alpha} \in \N_0^r, \quad
    \abso{g_{I,\beta}(\hat{\alpha}) - \tilde g_{I,\beta}^{(K)}(\hat{\alpha})} \leq M 2^{\frac{p}{2}} \jap{2na}^p n^{K+1} \jap{\hat{\alpha}}^{p+K+1}. 
  \end{equation}

  Let us now define an element $\tilde{f}^{(K)}$ of $\mathcal{P}_{n,K,a}$ by the formula
  \begin{equation}
    \label{e:f_tilde_shift}
    \tilde{f}^{(K)}(\alpha)
    := \sum_{\substack{I \subset \{1, \ldots, n\} \\ = \{i_1 < \ldots < i_r\}}}
      \sum_{\substack{(\beta_i)_{i \notin I} \\ \abso{\beta}_\infty < a}}
    \tilde g_{I,\beta}^{(K)}(\alpha_{i_1} - a, \ldots, \alpha_{i_r} - a) \prod_{i \in I} \1{\alpha_i \geq a} \prod_{i \notin I}
    \1{\alpha_i = \beta_i}. 
  \end{equation}
  By \cref{e:f_shift,e:f_tilde_shift} together with the bound \eqref{e:induc_f_beta_I}, for any $\alpha \in \N_0^n$,
  \begin{equation*}
    \abso{f(\alpha) - \tilde{f}^{(K)}(\alpha)}
    \leq M 2^{\frac{p}{2}+n} a^{n} \jap{2na}^p  n^{K+1} \jap{\alpha}^{p+K+1}
  \end{equation*}
  because there are $2^n$ terms in the sum over the $I \subseteq \{1, \ldots, n\}$, and always less than~$a^n$ possible
  choices for $\beta$. This is the claimed inequality.

  The coefficients of $\tilde{f}^{(K)}$ are linear combinations of the coefficients of the $\tilde g_{I, \beta}^{(K)}$. By
  \cref{lemm:taylor_zero}, these are themselves linear combinations of the values
  $\delta^{\hat{\mathbf{m}}} g_{I, \beta}(\z{r})$ for multi-indices $\hat{\mathbf{m}}$ of norm
  $\abso{\hat{\mathbf{m}}} \leq K$. By definition of $g_{I,\beta}$, these derivatives are derivatives of the form
  $\delta^{\mathbf{m}}f(\alpha)$ for multi-indices $\mathbf{m}, \alpha$ such that $\abso{\alpha}_\infty \leq a$ and
  $\abso{\mathbf{m}} \leq K$.
\end{proof}

\subsection{Proof of \cref{theo:volume_asympt_exp}}

We can now conclude with the proof of the asymptotic expansion,
\cref{theo:volume_asympt_exp}.

\begin{proof}
  Let $g \geq 0$, $n \geq 1$ be integers such that $2g-2+n>0$. Let $N \geq 0$ be a fixed
  order. By \cref{thm:derivative_coeff}, there exists constants
  $C_{n,N+1}, a_{N+1}$ such that
  \begin{equation*}
    \abso{\delta^{\mathbf{m}} c_{g,n}(\alpha)} \leq C_{n,N+1} \,  \jap{\alpha}^{N+1} \frac{V_{g,n}}{\jap{g}^{N+1}} 
  \end{equation*}
  for any multi-indices $\mathbf{m}, \alpha \in \N_0^n$ such that
  \begin{equation*}
    \abso{\mathbf{m}} = 2N+1 \quad \text{and} \quad
    \forall i, (m_i \neq 0 \Rightarrow \alpha_i \geq a_{N+1}).
  \end{equation*}
  This is exactly the hypothesis of \cref{lemm:taylor_shift}, for the
  parameters $K:=2N$, $p:=N+1$, $a := a_{N+1}$ and
  $M := C_{n,N+1} V_{g,n} / \jap{g}^{N+1}$. As a consequence, there exists an
  element $\tilde{c}_{g,n}^{(K)}$ of $\mathcal{P}_{n,K,a}$ such that for
  all $\alpha \in \N_0^n$,
  \begin{equation*}
    \abso{c_{g,n}(\alpha) - \tilde{c}_{g,n}^{(K)}(\alpha)} \leq C_{n,a,p,K} \, M \jap{\alpha}^{p+K+1}
  \end{equation*}
  or, in other words,
  \begin{equation}
    \label{e:c_gn_approx_K}
    c_{g,n}(\alpha) = \tilde{c}_{g,n}^{(2N)}(\alpha)
    + \O[n,N]{\jap{\alpha}^{3N+2} \frac{V_{g,n}}{\jap{g}^{N+1}}}.
  \end{equation}
  
  Let us now define, for all $\x \in \R_{\geq 0}^n$, a good candidate for the approximating function,
  \begin{equation*}
    F_{g,n}^{(N)}(\x)
    := \frac{1}{V_{g,n}} \sum_{\alpha \in \N_0^n} \tilde{c}^{(2N)}_{g,n}(\alpha) \prod_{i=1}^n \frac{x_i^{2 \alpha_i}}{2^{2 \alpha_i} (2
      \alpha_i+1)!} \cdot
  \end{equation*}
  Then, by \cref{e:c_gn_approx_K} and the definition of $V_{g,n}(\x)$ and $F_{g,n}^{(N)}(\x)$,
  \begin{equation*}
    \frac{V_{g,n}(\x)}{V_{g,n}} 
    = F_{g,n}^{(N)}(\x)
    + \O[n,N]{\sum_{\alpha \in \N_0^n} \frac{\jap{\alpha}^{3N+2}}{\jap{g}^{N+1}} \prod_{i=1}^n \frac{x_i^{2 \alpha_i}}{2^{2 \alpha_i} (2
        \alpha_i+1)!}}.
  \end{equation*}
  We can control this remainder by writing that $\jap{\alpha}^{3N+2} = \O[n,N]{1+\abso{\alpha}_\infty^{3N+2}}$ and
  singling out an index $i$ such that $\alpha_i = \abso{\alpha}_\infty$. Since, for any $D$, the polynomials
  $(p_i)_{0 \leq i \leq D}$ introduced in \cref{nota:p_k} are a basis of the set of polynomials of degree $\leq D$,
  we can express $\alpha_i^{3N+2}$ as a linear combination of $p_k(\alpha_i)$ for integers $k \leq 3N+2$. Using
  \cref{lemm:sum_p_k}, we obtain
  \begin{equation*}
    \frac{V_{g,n}(\x)}{V_{g,n}}
    = F_{g,n}^{(N)}(\x) + \O[n,N]{\frac{\jap{\x}^{3N+1}}{\jap{g}^{N+1}} \exp \div{x_1+\ldots+x_n}}.
  \end{equation*}
  
  Let us now prove that $\x \mapsto F_{g,n}^{(N)}(\x)$ has the claimed form. Note that by definition of the set
  $\mathcal{P}_{n,K,a}$, we can express the function $\alpha \mapsto \tilde{c}_{g,n}^{(K)}(\alpha)$ as a linear
  combination of functions of the form
  \begin{equation*}
    g_{I,\beta,\mathbf{k}}(\alpha)
    = \prod_{i \in I} p_{k_i}(\alpha_i) \prod_{i \notin I} \1{\alpha_i=\beta_i},
  \end{equation*}
  where $\abso{\mathbf{k}} \leq K$ and $\abso{\beta}_\infty < a_{N+1}$.
  By \cref{lemm:sum_p_k},
  \begin{align*}
    & \sum_{\alpha \in \N_0^n} g_{I,\beta,\mathbf{k}}(\alpha)
    \prod_{i=1}^n \frac{x_i^{2 \alpha_i}}{2^{2 \alpha_i} (2\alpha_i+1)!} \\
    & = \prod_{\substack{i \in I \\ k_i \text{ even}}}
    \frac{x_i^{k_i}}{2^{k_i}} \sinhc \div{x_i}
    \prod_{\substack{i \in I \\ k_i \text{ odd}}}
    \frac{x_i^{k_i-1}}{2^{k_i-1}}
    \cosh \div{x_i}
    \prod_{\substack{i \notin I}}
    \frac{x_i^{2 \beta_i}}{2^{2 \beta_i}(2 \beta_i+1)!} \cdot
  \end{align*}
  We therefore observe that $F_{g,n}^{(N)}$ is a linear combination of
  functions of the form
  \begin{equation*}
    \mathbf{x} \mapsto
    x_1^{m_1} \ldots x_n^{m_n}
    \prod_{i \in I_+} \cosh \div{x_i} \prod_{i \in I_-} \sinhc \div{x_i}
  \end{equation*}
  where $I_+$ and $I_-$ are disjoint subsets of $\{1, \ldots, n\}$, and $\mathbf{m}=(m_1, \ldots, m_n)$ is a multi-index
  containing only even entries, which was our claim.
  
  In order to bound the degree, we furthermore observe that
  \begin{equation*}
    \sum_{i \in I_+} m_i + \sum_{i \in I_-} (m_i+1) \leq K = 2N
    \quad \text{and} \quad
    \forall i \notin I_+ \cup I_-, m_i <  a_{N+1}.
  \end{equation*}
  The coefficients are linear combinations of the derivatives $\delta^{\mathbf{m}}c_{g,n}(\alpha)/V_{g,n}$ for
  multi-indices $\mathbf{m}, \alpha$ such that $\abso{\mathbf{m}} \leq K$ and $\abso{\alpha}_\infty \leq a_{N+1}$, which
  can therefore also be expressed in terms of the
  $c_{g,n}(\alpha)/V_{g,n}$ for $\abso{\alpha}_\infty \leq a_{N+1}+2N$.
\end{proof}

We now conclude by proving how \cref{theo:volume_asympt_exp} implies \cref{cor:exp_pow_g}.

\begin{proof}
  Let us first prove the existence of the asymptotic expansion.
  For any $I_+ \sqcup I_- \subset \{1, \ldots, n\}$, the coefficients of the approximating polynomial
  $P_{g,n}^{(N,I_{\pm})}$ can be written as linear combinations of the $c_{g,n}(\alpha)/V_{g,n}$ with
  $\abso{\alpha}_\infty \leq A_N$.  By \cite[Theorem
  4.1]{mirzakhani2015}, for any such $\alpha$, we can write
  \begin{equation*}
    \frac{c_{g,n}(\alpha)}{V_{g,n}}
    = \sum_{k=0}^N \frac{e_{n}^{(k)}(\alpha)}{g^k}
    + \O[n,N]{\frac{1}{g^{N+1}}}.
  \end{equation*}
  Note that the implied constant in the previous equation a priori depends on the multi-index $\alpha$, but since
  $\abso{\alpha}_\infty \leq A_N$ we can bound it uniformly with a constant depending only on $N$.
  Then, $P_{g,n}^{(N,I_\pm)}$ can be rewritten as
  \begin{equation}
    \label{e:Pgn_mirz_zog}
    P_{g,n}^{(N,I_\pm)}(\x)
    = \sum_{k=0}^N \frac{\tilde{Q}_{n}^{(k,N,I_\pm)}(\x)}{g^k}
    + \O[n,N]{\frac{\jap{\x}^{2N}}{g^{N+1}} \prod_{i \notin I_+ \cup I_-} \jap{x_i}^{a_{N+1}}},
  \end{equation}
  where $\tilde{Q}_{n}^{(k,N,I_\pm)}$ are polynomial functions independent of $g$. The dependency of the remainder
  w.r.t. $\x$ in the previous expression is obtained by the bound on the degrees of $P_{g,n}^{(N,I_\pm)}$ presented in
  \cref{rem:deg}.  We then define, for each integer $k$, the function
  \begin{equation*}
    \tilde{f}_n^{(k,N)}(\x) := \sum_{I_+ \sqcup I_- \subset \{1, \ldots, n\}} \tilde{Q}_n^{(k,N,I_\pm)}(\x)
    \prod_{i \in I_+} \cosh \div{x_i} \prod_{i \in I_-} \sinhc \div{x_i}.
  \end{equation*}
  \Cref{e:thm_main,e:Pgn_mirz_zog} together with the fact that
  \begin{equation*}
    \prod_{i \notin I_+ \sqcup I_-} x_i^{a_{N+1}} \prod_{i \in I_+} \cosh \div{x_i} \prod_{i \in I_-} \sinh \div{x_i}
    = \O[N,n]{\exp \div{x_1+ \ldots + x_n}}
  \end{equation*}
  imply that these approximating functions satisfy \eqref{e:fn_approx}, i.e. for all $\x \in \R_{\geq 0}^n$,
  \begin{equation*}
    \frac{V_{g,n}(\x)}{V_{g,n}}
    = \sum_{k=0}^N \frac{\tilde{f}_{n}^{(k,N)}(\x)}{g^k}
    + \O[N,n]{\frac{\jap{\x}^{3N+1}}{g^{N+1}} \exp \div{x_1 + \ldots + x_n}}.
  \end{equation*}

  We now observe that, for any fixed $\x$, the previous equation is an asymptotic expansion of $V_{g,n}(\x)/V_{g,n}$ in
  powers of $g$, and its coefficients are therefore uniquely defined. In particular, for any $N'<N$ and any
  $k \in \{0, \ldots, N'\}$, $\tilde{f}_n^{(k,N)} (\x) = \tilde{f}_n^{(k,N')}(\x)$, and the number
  $\tilde{f}_n^{(k,N)}(\x)$ does not depend on the order of approximation $N$ and can be denoted more simply as
  $f_n^{(k)}(\x)$.
  
  Finally, the decomposition \eqref{eq:fn_sinh} of $f_n^{(k)}$ is uniquely defined because the family of functions of
  the form
  \begin{equation*}
    \x \in \R_{\geq 0}^n \mapsto x_1^{m_1} \ldots x_n^{m_n} \prod_{i \in I_+} \cosh \div{x_i} \prod_{i
      \in I_-} \sinhc \div{x_i}
  \end{equation*}
  for $m_1, \ldots, m_n \geq 0$ and
  $I_+ \sqcup I_- \subseteq \{1, \ldots, n\}$ is free.
\end{proof}

\bibliographystyle{alpha}
\bibliography{bibliography}

\end{document}